\documentclass[12pt]{amsart}

\newif\iftexpad
\texpadfalse

\usepackage{graphicx}
\usepackage[final]{epsfig}

\usepackage{amsmath}
\usepackage{amsfonts}
\usepackage{latexsym}
\usepackage{amssymb}
\usepackage{url}
\usepackage[hidelinks]{hyperref}
\usepackage{color}
\usepackage{comment}

\usepackage{amsfonts}
\usepackage{amsthm}
\usepackage{a4wide}
\usepackage{marginnote}

\newtheorem{lemma}{Lemma}[section]
\newtheorem{proposition}{Proposition}[section]
\newtheorem{remark}{Remark}[section]

\newtheorem{theorem}[lemma]{Theorem}

\newtheorem{corollary}{Corollary}[section]
\newtheorem{conjecture}{Conjecture}[section]
\newtheorem{question}[lemma]{Question}

\newcommand{\R}{{\mathbb R}}

\iftexpad
\newcommand{\cref}[1]{\ref{#1}}
\else
\usepackage{thmtools}
\usepackage{nameref}
\usepackage[noabbrev,capitalise]{cleveref}
\crefdefaultlabelformat{#2#1#3} 
\creflabelformat{equation}{#2#1#3} 
\crefname{equation}{Equation}{Equations}
\crefname{figure}{Figure}{Figures}
\crefname{tabular}{Table}{Tables}
\crefname{section}{Section}{Sections}
\crefname{proposition}{Proposition}{Propositions}
\crefname{conjecture}{Conjecture}{Conjectures}
\crefname{corollary}{Corollary}{Corollaries}
\crefname{remark}{Remark}{Remarks}
\Crefname{equation}{Eq.}{Eqs.}
\Crefname{figure}{Fig.}{Figs.}
\Crefname{tabular}{Tab.}{Tabs.}
\Crefname{section}{Sec.}{Secs.}
\Crefname{proposition}{Prop.}{Props.}
\Crefname{conjecture}{Conj.}{Conjs.}
\Crefname{corollary}{Cor.}{Cors.}
\Crefname{remark}{Rem.}{Rems.}

\fi

\usepackage{subfig}
\usepackage{caption}
\usepackage{xspace}
\newcommand{\quadex}{Quad\xspace}
\newcommand{\penthouseex}{Penthouse\xspace}
\newcommand{\hexhouseex}{Hexhouse\xspace}

\title{Polygonal symplectic billiards}

\author{Peter Albers}
\author{Gautam Banhatti}
\author{Filip Sadlo}
\author{Richard Schwartz}
\author{Serge Tabachnikov}

\address{Peter Albers,
 Mathematisches Institut and Interdisciplinary Center for Scientific Computing (IWR),
 Ruprecht-Karls-Universit\"at Heidelberg,
 Germany}
\email{peter.albers@uni-heidelberg.de}

\address{Gautam Banhatti,
 Mathematisches Institut,
 Ruprecht-Karls-Universit\"at Heidelberg,
 Germany}
\email{gautam@posteo.de}

\address{Filip Sadlo,
 Interdisciplinary Center for Scientific Computing (IWR),
 Ruprecht-Karls-Universit\"at Heidelberg,
 Germany}
\email{sadlo@uni-heidelberg.de}

\address{Richard Schwartz,
Department of Mathematics,
Brown University,
USA}
\email{res@math.brown.edu}

\address{Serge Tabachnikov,
 Department of Mathematics,
 Pennsylvania State University,
USA}
\email{tabachni@math.psu.edu}
\date{\today}

\begin{document}

\begin{abstract}
In this article, we study polygonal symplectic billiards. We provide new results, some of which are inspired by numerical investigations. In particular, we present several polygons for which all orbits are periodic. We demonstrate their properties and derive various conjectures using two numerical implementations.
\end{abstract}

\maketitle

\pagestyle{headings}
\markleft{\sc Peter Albers, Gautam Banhatti, Filip Sadlo, Richard Schwartz, Serge Tabachnikov}

\section{Introduction} \label{sect:intro}

\begin{figure}[t]
  \centering%
  \includegraphics[width=2in]{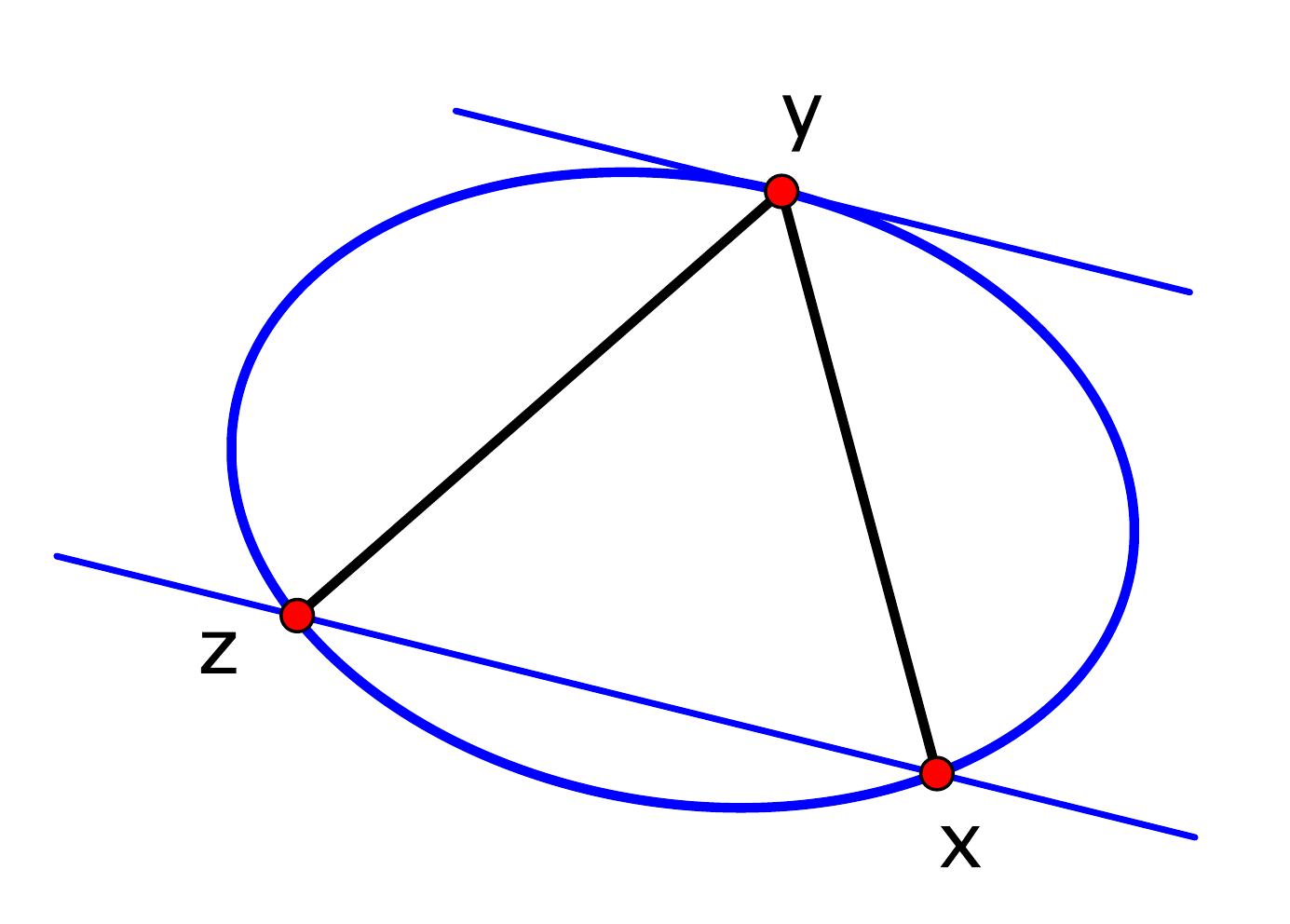}%
  \caption{\label{fig:reflection}
    The symplectic billiard reflection: $xy$ reflects to $yz$ if $xz$ is parallel to the tangent line of the curve at point $y$.
  }
\end{figure}

Planar symplectic billiard is a discrete-time dynamical system on oriented chords of a piecewise smooth convex closed curve (billiard table) in the plane depicted in \cref{fig:reflection}. Symplectic billiards were introduced by Albers and Tabachnikov~\cite{AT}. The name symplectic billiards is due to the fact that they can be defined in linear symplectic space. In the plane, symplectic billiards commute with affine transformations.

A symplectic billiard table may be a polygon. In this case, the reflection is not defined if the head of an oriented chord is a vertex of the polygon or when the head and the tail of an oriented chord belong to parallel sides. In this article, we only consider convex polygons, (even though symplectic billiards can also be defined on non-convex polygons).

So far, two classes of polygons were considered~\cite{AT}, affine-regular polygons and trapezoids. In both cases, all symplectic billiard orbits are periodic. In this paper, we describe other families of polygons with this property. We call them periodic polygonal symplectic billiards.

Let us briefly mention two other, much better known, classes of polygonal billiards: Euclidean billiards and outer billiards (see, e.g., the book \cite{Ta} for a survey). In the former billiards, periodic trajectories appear in 1-parameter families of mutually parallel trajectories, but trajectories with different initial directions, no matter how close, will eventually diverge and hit different sides of the polygon. The celebrated Ivrii conjecture states that the set of periodic orbits of a planar billiard has zero phase area. However, the example of an equilateral right spherical triangle shows that in spherical geometry billiards all of whose orbits are closed exist. Moreover, in outer billiards, which is played outside of the curve, it is possible for all trajectories to be periodic: this happens for all lattice polygons, see again the book \cite{Ta}.

In this article some of our results are proof-based while others are driven by numerical experiments. Based on that we formulate several conjectures. For our numerical investigations we utilized two research codes developed in the context of this paper, one for effective determination of periodicity and one for interactive exploration of phase-space structure. Beyond that we only aware of implementations in this field by Boshe-Ploes et al.\ \cite{BNAL} and Raymond Friend in his honors thesis at Pennsylvania State Universtiy.

\section{General facts about polygonal symplectic billiards} \label{sect:gen}

In this section, we recall known and prove some new facts about polygonal symplectic billiards.

\subsection{Phase space and phase area.}$ $\\[1ex] 
Let ${\bf P}$ be an $n$-gon with vertices $P_1,P_2,\ldots,P_n$, oriented counterclockwise. We define the vectors $v_i=P_{i+1}-P_i$, where the indices are understood cyclically and denote by $v_i \times v_j$ the set of chords whose tail is in the interior of the side~$P_i P_{i+1}$ and whose head is in the interior of the side $P_j P_{j+1}$. We use bracket $[\cdot , \cdot]$ to denote the determinant of two vectors. 

Let $T$ be the symplectic billiard map defined as in \cref{fig:reflection}. It is piecewise continuous. Its phase space is the union of the sets $v_i \times v_j$, $i,j=1,\dots,n$, with $i\neq j$ and $v_i$ and $v_j$ not being parallel. 
After parameterizing the perimeter of ${\bf P}$, the phase space is represented by a square, tiled by the rectangles $v_i \times v_j$. The rectangles $v_i \times v_j$ with $v_i$ and $v_j$ being parallel, along with the squares $v_i \times v_i$, are excised (they are represented by black squares in the pictures below). In this representation, the first coordinate describes the position of the tail of a chord, and the second coordinate the position of its head. The map $T$ is continuous in each rectangle $v_i \times v_j$, see \cref{lm:area}.

The phase space has an involution that reverses the direction of a chord. This ``time reversal" involution conjugates $T$ and $T^{-1}$. We reduce the phase space by half by considering only the rectangles $v_i \times v_j$ with $[v_i,v_j]>0$ and denote this  space by $\Phi_{\bf P}$. Further, define a piece-wise constant area form on the phase space by declaring that the total area of a rectangle $v_i \times v_j$ is $[v_i,v_j]$. We parametrize each side by arc-length and denote the corresponding coordinates by $x$, $y$, $z$ etc. Then, if $\alpha$ is the angle between $v_i$ and $v_j$, and $dx$ and $dy$ are the respective oriented length elements on these sides, then the area form equals $\sin\alpha\ dx \wedge dy$.

\begin{lemma} \label{lm:area}
The map $T$ is area preserving. It has the form
$$
T: (x,y) \mapsto (y, z=ax+b),\ \ x\in P_i P_{i+1}, y\in P_j P_{j+1}, z\in P_k P_{k+1},
$$
with $a=-\frac{\sin\alpha}{\sin\beta}$ and $b$ depending on $i,j,k$ (but not on $y$), see \cref{fig:distortion}. 
\end{lemma}

\begin{figure}[t]
  \centering%
  \includegraphics[width=3.2in]{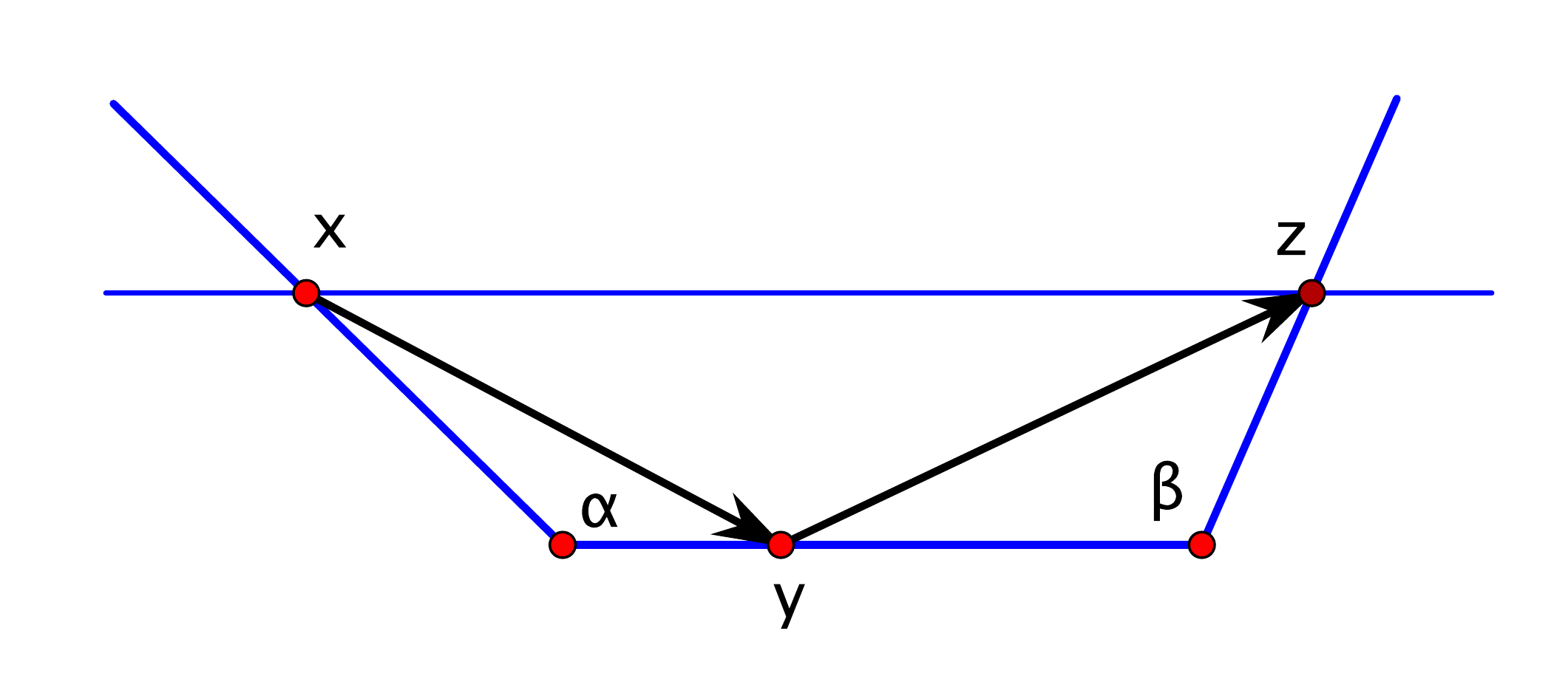}%
  \caption{\label{fig:distortion}
    Distortion of length under projection, see \cref{lm:area}.
  }
\end{figure}

\begin{proof}
Consider an instance of a reflection, \cref{fig:distortion}. The projection of the $i$th side on the $k$th side along the $j$th side reverses the orientation. This projection is an affine map~$x \mapsto z$ that does not depend on $y$, as long as $y$ stays on the $j$th side. The projection distorts the length by the ratio $\sin\alpha / \sin\beta$, namely, $\sin\alpha\ dx = -\sin\beta\ dz$. It follows that $\sin\alpha\ dx \wedge dy = \sin\beta\ dy \wedge dz$, as needed. 
\end{proof}

\begin{remark}\label{rmk:symplectic_billiard_is_an_indefinite_isometry}
{\rm
The phase space of polygonal symplectic billiards has a $T$-invariant area form~$\omega$ and a $T$-invariant 2-web of vertical and horizontal lines. These two structures determine a sign-indefinite quadratic form as follows.  Given a vector $V$, let $V_1$ and $V_2$ be its horizontal and vertical components with respect to the 2-web, and set $g(V)= \omega(V_1,V_2)$. 
We obtain a pseudo-Euclidean metric $g$, and it follows that  $T$ is a piecewise isometry relative to this pseudo-Euclidean metric. 
}
\end{remark}

We now give another interpretation of the phase area.
Denote by $-{\bf P}$ the reflection of ${\bf P}$ about the origin. Then the 
{\it difference body} $D({\bf P})$ of a convex body ${\bf P}$ is the Minkowski sum of ${\bf P}$ with $-{\bf P}$. In other words, the difference body is centered at the origin and is formed by the vectors that connect pairs of points of ${\bf P}$. For example, the difference body of a triangle is an affine-regular hexagon, and the difference body of a square is a square twice as large.

Let ${\bf P}$ be a convex plane polygon, the symplectic billiard table. This induces a map $f: \Phi_{\bf P} \to D({\bf P})$ that sends chords of ${\bf P}$ to points of its difference body. We equip the latter with the area form induced from that in the plane. 

\begin{lemma} \label{lm:diffbody}
The map $f$ is an area-preserving bijection of the interior of $\Phi_{\bf P}$ to an open dense subset of $D({\bf P})$.
\end{lemma}

\begin{proof} Let us construct an inverse of the map $f$ on an open dense part of $D({\bf P})$.  For this we assume for the moment that ${\bf P}$ is strictly convex with smooth boundary $\gamma$, oriented counter-clockwise. An affine diameter of ${\bf P}$ is a chord of $\gamma$ with parallel tangent lines at its end points.

Let vector $w$ be a nonzero vector in the interior of $D({\bf P})$. We would like to represent $w$ as the vector $AB$, where  $A\neq B$ are in the interior of ${\bf P}$. If one, or both, points $A, B$ lie on $\gamma$, one can parallel-translate the segment $AB$ so that both points are inside ${\bf P}$. Indeed, the only situation when such a translation does not exist, is when $AB$ is an affine diameter of ${\bf P}$, as illustrated in \cref{convex}. But, then $w$ would lie on the boundary of the difference body, the case that we already excluded.

Consider the oriented line through $AB$ and move it to the right (with respect to the orientation of the plane). By continuity and strict convexity, there will be a unique moment when the intersection points, $A',B'$, of the moving line with $\gamma$  are such that $A'B'=w$. Let $u$ and $v$ be the oriented tangent vectors of $\gamma$ at $A'$ and $B'$.  We claim that $[u,v] \neq 0$. Indeed, if $[u,v] = 0$, then either $v=u$ or $v=-u$. In the former case, $B'=A'$ and $w=0$, and in the latter case, $A'B'=AB$ is an affine diameter. Both case are  excluded. The punctured interior of $D({\bf P})$ is connected, hence $[u,v]$ has a constant sign, and it is easy to see that it is positive (see \cref{convex}). Thus, $A'B' \in \Phi_{\bf P}$, and the map $w \mapsto A'B'$ is the inverse of $f$.

\begin{figure}[t]
\centering%
\includegraphics[width=1.8in]{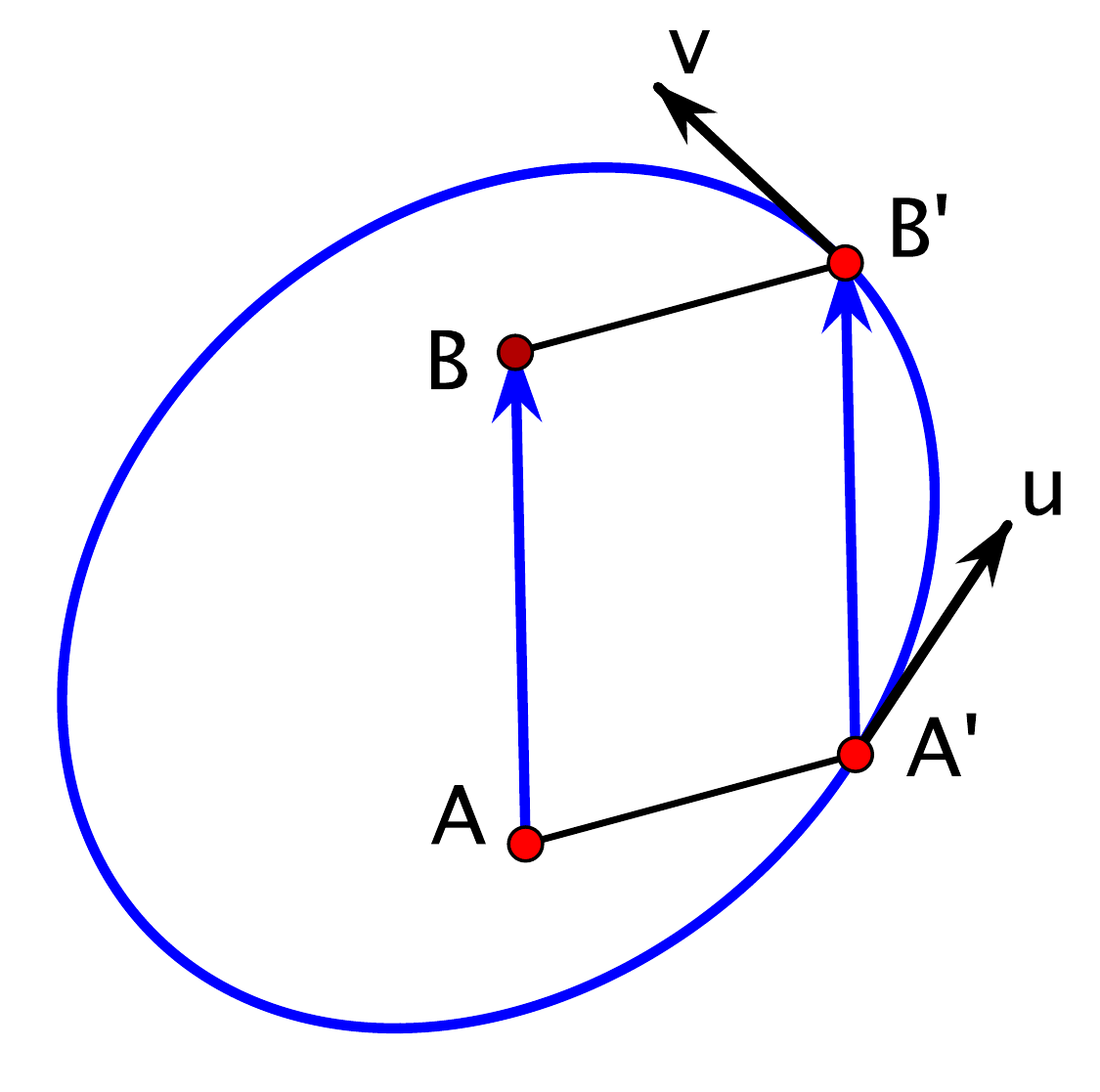}%
\caption{\label{convex}
Moving a vector to the boundary.
}
\end{figure}

If ${\bf P}$ is a convex polygon, the same construction applies with the following adjustments. In addition to the boundary of  $D({\bf P})$ and the origin, we also remove from it the vectors that are equal to $tv_i$ for some index~$i$ and $t\in[0,1]$, the vectors whose endpoints are on parallel sides of ${\bf P}$ (if any), and the vectors for which point $A'$ or $B'$ is a vertex of the polygon. These sets are 1-dimensional, and $f$ is a bijection of the interior of $\Phi_{\bf P}$ to their complement.

Concerning the area-preserving property, let $x$ and $y$ be arc length parameters on the sides on which points $A'$ and $B'$ lie, and let $u$ and $v$ be the unit orienting vector along these sides, respectively. Then, locally, the map $f$ is given by $(x,y) \mapsto yv-xu \in {\R}^2$. The induced area form is $[u,v]dx \wedge dy$, as needed.
\end{proof}

\begin{remark}
{\rm
The total phase space area of the symplectic billiard in a strictly convex plane domain with smooth boundary equals the area of its difference body, \cite{AT}. The proof involves manipulations with the support function of the body.
}
\end{remark}

\subsection{Symbolic dynamics and tiles}$ $\\[1ex]
As before, we label the sides of the polygon $1,\ldots, n$, and assign to each orbit of the symplectic billiard map $T$ its symbolic orbit, the bi-infinite sequence of the labels of the sides that are visited by the orbit. A periodic billiard orbit has a periodic symbolic orbit. 

Define a {\it tile} as the set of phase points with the same periodic symbolic orbit. 
The {\it discontinuity set} consists of the the phase points for which some iteration of $T$, in the future or in the past, is not defined, that is, whose orbit ends up at a vertex or starts at a vertex of the polygon. The set of chords, one of whose endpoints is a vertex and another lies on a side, is either a horizontal or a vertical segment. \cref{lm:area} implies that the discontinuity set is a union of horizontal and vertical segments. Thus, its complement is the union of  tiles. 

\begin{lemma} \label{lm:connected}
The tiles are phase rectangles, possibly degenerate (segments or points). In particular, every tile is connected. \\
If a tile is a genuine rectangle, that is, has a positive phase area, then its symbolic orbit is periodic. Furthermore, 
every orbit in this tile is periodic. More precisely, let $M$ be a tile of positive area with a periodic symbolic orbit of period $n$. Then $T^n$ maps $M$ to itself, and the return map $T^n$ has either order 4, or order 2, or it is the identity.
\end{lemma}

\begin{proof}
Let $(\ldots, i_0, i_1,i_2, \ldots)$ be a symbolic orbit. The phase points with the symbolic coding $(i_0,i_1)$ form the set $v_{i_0}\times v_{i_1}$, the points with the coding $(i_0,i_1,i_2)$ form the set $(v_{i_0}\times v_{i_1}) \cap T^{-1}(v_{i_1}\times v_{i_2})$, and so on. The preimages and images of rectangles with vertical and horizontal sides are rectangles with vertical and horizontal sides, and the intersection of a finite number of such rectangles is again a rectangle of this kind, cf.\ \cref{lm:area}. An infinite intersection is still a rectangle, possibly a degenerate one. 

Assume that a tile $M$ has positive area. Since the total phase area is finite and $T$ is area preserving, there exist $i>j$ such that the tiles $T^i(M)$ and $T^j(M)$ intersect. Hence, $T^{i-j}(M)$ intersects $M$, and since $M$ is a tile, it follows that $T^{i-j}(M)=M$. Therefore the symbolic orbit of $M$ is  $(i-j)$-periodic. The tile $M$ is a rectangle, and the return map $T^{i-j}$ is an orientation-preserving affine isomorphism of this rectangle. Hence, this map is conjugated to a rotation of a square through an angle that is a multiple of $\pi/2$, that is, either $\pi/2$ (order 4), or $\pi$ (order 2), or $2\pi$ (the identity).
\end{proof}

\subsection{Periodic trajectories.} $ $\\[1ex] 
Call a periodic trajectory in a polygon ${\bf P}$ {\it stable} if this trajectory persists under every sufficiently small perturbation of ${\bf P}$. For example, the 3-periodic orbits in a triangle that connects the mid-points of its sides is stable, whereas a  4-periodic trajectory in a square is not stable: it can be destroyed by an arbitrary small perturbation of the square.  

A periodic trajectory is called {\it isolated} if its tile has zero area. 

\begin{proposition} \label{prop:periodic}
An isolated periodic orbit is stable. In addition,
\begin{itemize}
\item if $n$ is odd, then an $n$-periodic orbit is stable. However, it is never isolated, and the return map to its tile has order 4;
\item if $n$ is even but not divisible by four, then an $n$-periodic orbit is stable. If it is not isolated, then the return map to its tile has order 2;
\item if $n$ is divisible by four and the quotient $\lambda$, (which is defined in \eqref{eq:linearized_return_map_case_n=4k} in the proof below) is different from 1,  
then the respective periodic orbit is stable. If $\lambda =1$, then the return map of the respective tile is the identity. In this case, as indicated by our numerical experiments, the orbit may be stable or unstable.
\end{itemize}
\end{proposition}

\begin{remark}
{\rm
The proof of \cref{prop:periodic} relies on the simple fact that an orbit is stable, if the differential of the return map does not have 1 as an eigenvalue. In our case this map is a composition of rather explicit affine transformations which leads to the description in \cref{prop:periodic}.
}	
\end{remark}

\begin{proof}
Let $x_1,\ldots,x_n$ be a periodic trajectory. For every $i$, the side $L_i$ containing point $x_i$ is parallel to  $x_{i-1},x_{i+1}$.  Let $\alpha_i$ be the angle between $L_i$ and $L_{i+1}$. 

Let $n$ be odd, and let $y_1 y_2$ be a chord sufficiently close to $x_1 x_2$. 
We trace the evolution of odd-numbered and even-numbered points $y_i$ separately.
After $n$ reflections, $y_1$ returns to the line $L_1$ as $y_{1+n}$ and $y_2$ returns to the line $L_2$ as $y_{2+n}$ and similarly after $2n$ reflections, $y_1$ returns to the line $L_1$ as $y_{1+2n}$ and $y_2$ returns to the line $L_2$ as $y_{2+2n}$. Each of these return maps reverses the orientation and preserves the length, because the distortion of the length  equals
$$
\frac{\sin \alpha_1}{\sin \alpha_2} \frac{\sin \alpha_3}{\sin \alpha_4} \cdots \frac{\sin \alpha_{2n-1}}{\sin \alpha_{2n}} = 1
$$
(every angle appears twice, once in the numerator and once in the denominator). That is, the return maps to the sides $L_1$ and $L_2$ are the reflection in points $x_1$ and $x_2$, respectively. See \cref{parity}.

\begin{figure}[t]
\centering%
\includegraphics[width=3.2in]{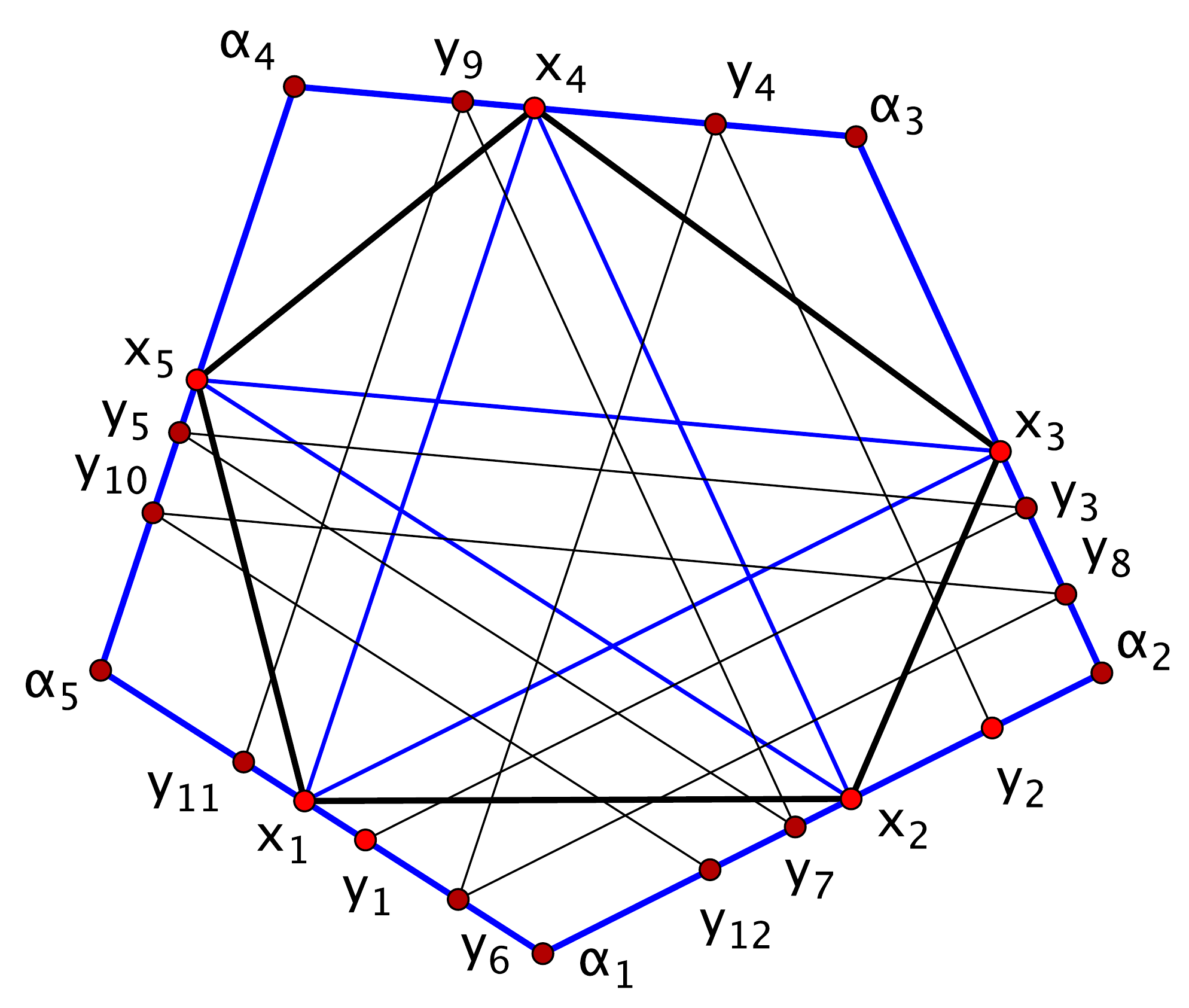} %
\caption{\label{parity}
A 5-periodic trajectory $x_1,\ldots,x_5$ and a nearby trajectory $y_1,y_2,\ldots$. Points $y_1$ and $y_{11}$ are symmetric with respect to point $x_1$, and points $y_2$ and $y_{12}$ are symmetric with respect to point $x_2$.
}
\end{figure}

It follows that the phase point $y_1 y_2$ is $4n$-periodic. These phase points form a tile, this tile returns to itself after $n$ iterations, and order of the return map $T^n$ is four. The phase point $x_1 x_2$ is the center of this tile, and it is
a hyperbolic fixed point of $T^n$ with the eigenvalues $\pm\sqrt{-1}$. A small perturbation of a polygon does not destroy such a fixed point, hence the perturbed polygon also has an $n$-periodic symplectic billiard trajectory.

Now let $n=2m$ be even. Arguing in the same way, the first return of point $y_1$ to line $L_1$ occurs after $n$ reflections, and likewise for point $y_2$. 

The distortion of the length on $L_1$ equals
\begin{equation} \label{eq:linearized_return_map_case_n=4k}
\frac{\sin \alpha_1}{\sin \alpha_2} \frac{\sin \alpha_3}{\sin \alpha_4} \cdots \frac{\sin \alpha_{n-1}}{\sin \alpha_n} =: \lambda,
\end{equation}
and for $L_2$, the distortion equals $1/\lambda$. The orbit is not isolated if and only if $\lambda=1$, cf.\ \cref{lm:area}.

If $m$ is odd, each of these return maps reverses the orientation of the line, and is a homothety with a negative coefficient. Hence the fixed point persists under a sufficiently small perturbation of the polygon.
It follows that the $n$-periodic orbit is stable. If the orbit is not isolated, and $\lambda=1$, then the return map of the respective tile is a reflection in a point, that has order two. 

If $n$ is a multiple of four and $\lambda \neq 1$, then the respective periodic orbit is hyperbolic, with one attracting and one repelling direction. Therefore it is stable. If $\lambda =1$, then the return map of the respective tile is the identity. The orbit may be either stable or unstable.
\end{proof}

\section{The \quadex} \label{sect:quad}

The fist periodic polygonal symplectic billiard table that we discovered is a quadrilateral that we call {\it the \quadex}, see \cref{quad_space}. All phase points are periodic with two periods, 20 and 36. See \cref{quad_space} for the phase space colored according to period. We now analyze the dynamics.   


\begin{figure}[t]
  \newcommand{\figheight}{0.33\linewidth}%
  \centering%
  \subfloat[\label{fig:quad_space-config}]{%
    \includegraphics[height=\figheight]{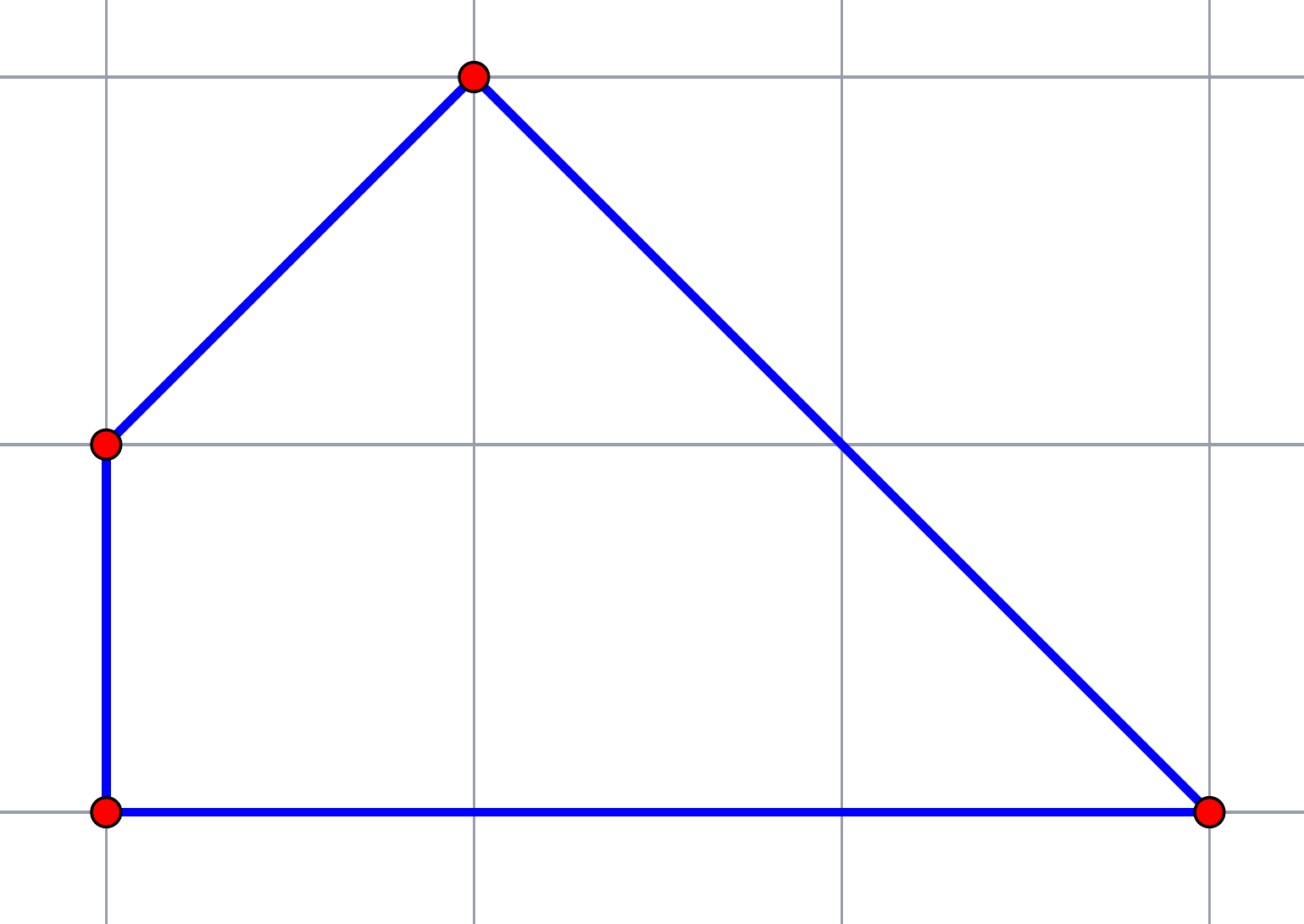}%
  }%
  \hfil%
  \subfloat[\label{fig:quad_space-phase}]{%
    \includegraphics[height=\figheight,trim={0 0 2pt 0},clip]{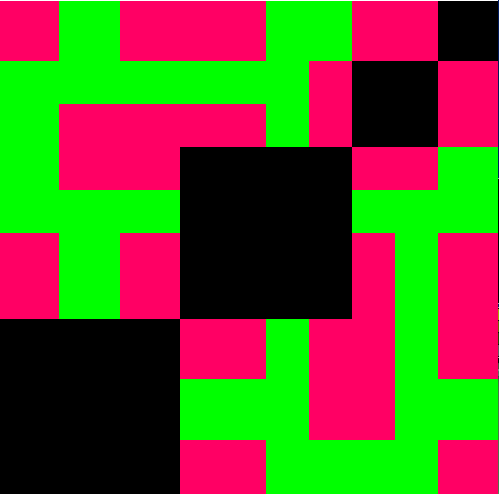}%
  }%
  \caption{\label{quad_space}
    The \quadex in configuration space~\protect\subref{fig:quad_space-config} and its phase space~\protect\subref{fig:quad_space-phase}. Green parts of the phase space are 36-periodic, and the red ones 20-periodic. 
  }
\end{figure}

\begin{theorem} \label{thm:quad}
All orbits in the \quadex are periodic with periods 20 and 36. The structure of the orbits of the periodic tiles is as follows.
\begin{itemize}
\item One orbit consisting of the tiles that return to themselves after 10 iterations, with the return map having 
order 2.
\item One orbit consisting of the tiles that return to themselves after 9 iterations, with the return map having 
order 4.
\end{itemize}
\end{theorem}

\begin{proof}
Consider \cref{quad_marked}. We will describe the evolution of two phase rectangles under the symplectic billiard map $T$, the rectangles $AB\times BC$ and $AB\times CE$.

\begin{figure}[t]
\centering%
\includegraphics[width=2.7in]{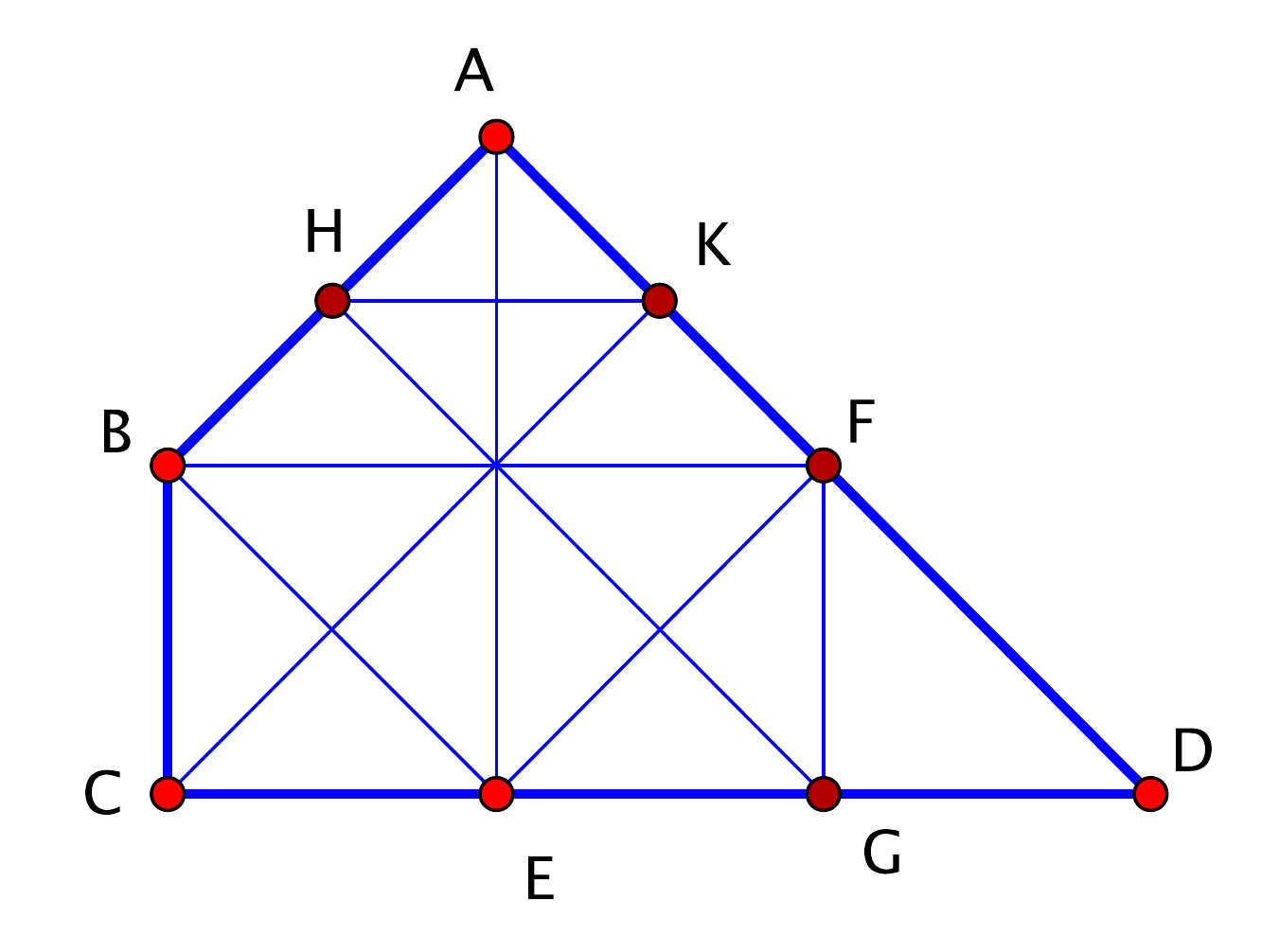} %
\caption{\label{quad_marked}
The \quadex with marked points on the sides.
}
\end{figure}

\begin{figure}[t]
  \newcommand{\figheight}{0.344\linewidth}%
  \centering%
  \subfloat[\label{fig:centers-ten}]{%
    \includegraphics[height=\figheight]{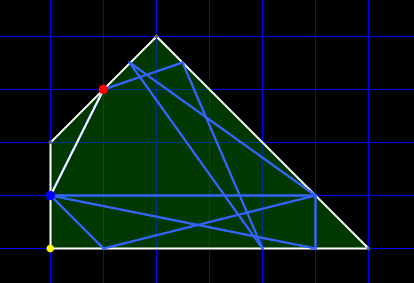}%
  }%
  \hfill%
  \subfloat[\label{fig:centers-nine}]{%
    \includegraphics[height=\figheight]{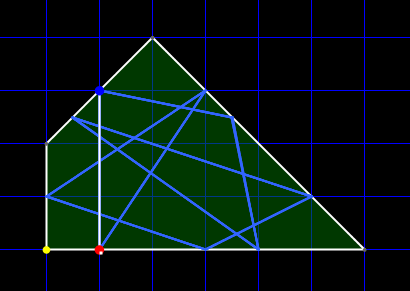}%
  }%
  \caption{\label{centers}
    The 10-periodic~\protect\subref{fig:centers-ten} and 9-periodic~\protect\subref{fig:centers-nine} trajectories in the \quadex.
  }
\end{figure}

In the first case, we have
\begin{equation*}
\begin{aligned}
AB\times BC \to BC\times CE \to CE\times DF \to DF\times BC \to BC\times GD \to GD \times DF \to\\
 DF\times AH \to AH\times ED \to ED\times KA \to KA\times AB \to AB\times BC.
\end{aligned}
\end{equation*}
Note that the orbit of the rectangle $AB\times BC$ is never split by a discontinuity line, and that it returns to itself after 10 iterations. Following a small segment in $AB\times BC$ reveals that this return map has order two, and its second iteration yields 20-periodic points.

In the second case, we have
\begin{equation*}
\begin{aligned}
AB\times CE \to CE\times FA \to FA\times BC \to BC\times EG \to EG\times DF \to\\
 DF\times HB \to HB\times ED \to ED\times FK \to FK\times AB \to AB\times CE.
\end{aligned}
\end{equation*}
Once again, the rectangle is never split and it returns back after 9 iterations. As above one sees that the return map has order four, and its fourth iteration yields 36-periodic points.

Assuming that the grid squares in \cref{quad_space} are unit, one calculates the area of the phase space (with respect to the symplectic form on phase space) to be equal to 19, whereas the phase areas of the rectangles $AB\times BC$ and $AB\times CE$ are unit. The orbit of the former rectangle has area 10, and that of the latter has area 9. Hence the whole phase space is tiled by these rectangles, proving that every phase point is either 20- or 36-periodic. 

The centers of the rectangles $AB\times BC$ and $AB\times CE$ are, respectively, 10- and 9-periodic trajectories, see \cref{centers}.
\end{proof}

\begin{remark}\label{rm:slow_fast}
{\rm
\cref{perturbation} shows the phase portrait of the symplectic billiard map of a small perturbation of the \quadex. In accordance with \cref{prop:periodic}, the 10-periodic tile gives rise to an isolated 10-periodic hyperbolic orbit, whereas the 9-periodic orbit persists and is surrounded by a periodic tile. The points close to the hyperbolic 10-periodic orbit exhibit a kind of slow-fast dynamics. They travel slowly along the hyperbolas while jumping fast between tiles, as is clearly seen in the figure. 
}
\end{remark}

\begin{figure}[t]
\centering%
\includegraphics[height=4in,trim={3pt 0 3pt 0},clip]{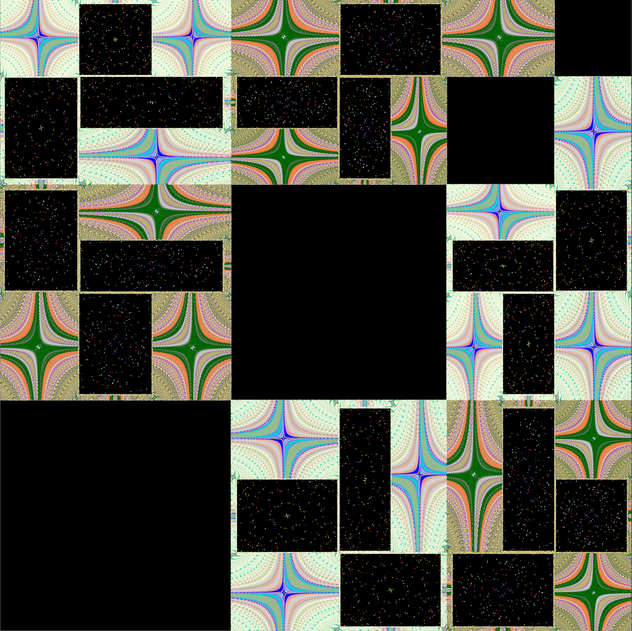}%
\caption{\label{perturbation}
The phase portrait of a small perturbation of the \quadex.
}
\end{figure} 

\section{The \penthouseex} \label{sect:pent}

The next periodic polygonal symplectic billiard table in our collection is a pentagon obtained by placing a triangle on top of a parallelogram. Applying an affine transformation, we normalize the parallelogram to be a unit square, and we call the resulting pentagon {\it the \penthouseex} (a pentagon that looks like a house), see \cref{fig:pent}.  

\begin{figure}[t]
\centering%
\includegraphics[height=1.8in]{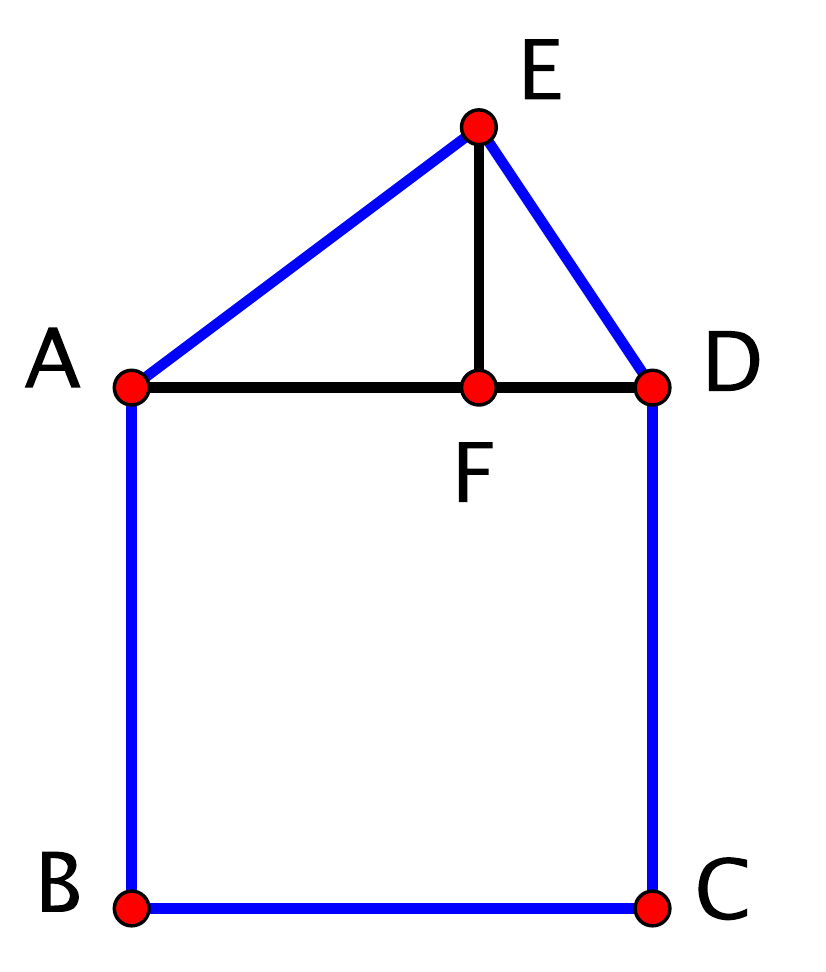}%
\caption{\label{fig:pent}
The \penthouseex.
}
\end{figure} 

The affine moduli space of these pentagons is 2-dimensional (while the affine moduli space of all pentagons is 4-dimensional). As parameters $(a,b)$, one can choose the elevation of the ``roof"
$a=|EF|$ and its horizontal displacement $b=|AF|$. Then $a$ is a positive real, and $0<b<1$.  Note that, in the limit $b=0$ or $b=1$, we obtain a trapezoid.

A straightforward calculation yields the following result.

\begin{lemma} \label{lm:pentareas}
The following table shows 9 rectangles whose union is the phase space together with their phase areas.
$$
\begin{array}{|c|c|c|c|c|}
\hline
AB\times BC & BC\times CD & BC \times DE & CD \times DE & CD\times EA \\ 
\hline
1 & 1 & a & 1-b & b\\
\hline	
\end{array}
$$
$$
\begin{array}{|c|c|c|c|}
\hline
DE\times EA & DE\times AB & EA\times AB & EA\times BC\\
\hline
a & 1-b & b & a\\
\hline
\end{array}	
$$
\end{lemma}

Recall a result concerning the trapezoids, \cite{AT}. Let $u>v$ be the lengths of the parallel sides of a trapezoid. Define its {\it modulus} as $\lfloor u/(u-v)\rfloor\in \mathbb{Z}$; this is an affine invariant. A trapezoid is {\it generic} if $u/(u-v) \not\in \mathbb{Z}$. The result that we need is as follows: {\it all orbits in a trapezoid are periodic, and if the modulus of a generic trapezoid is $m$, then the periods are $16m-4, 16m+4$, and $16m+12$}.

Define the modulus of a \penthouseex similarly:
$$
m = \left\lfloor\frac{a+1}{a} \right\rfloor.
$$
When a \penthouseex degenerates to a trapezoid, its modulus becomes that of the trapezoid. A generic \penthouseex is defined similarly: $\frac{a+1}{a} $ is not an integer. We call a \penthouseex {\it tall} if $a>1$, that is, if
$m=1$.

\begin{conjecture} \label{conj:pent}
All orbits in a generic \penthouseex are periodic with the periods equal to $16m-4, 16m+4$, and $16m+12$, in particular, the periods do not change if one moves the roof horizontally (but the symbolic orbits may change). 
Under the bifurcation $m \mapsto m+1$, the tiles with the largest period $16m+12$ survive and become the tiles with the smallest period $16m+12 = 16(m+1)-4$; the tiles with the periods $16m-4$ and  $16m+4$ die and the tiles with periods $16m+20$ and $16m+28$ are born.
\end{conjecture}

\begin{remark}
{\rm
We point out that, if \cref{conj:pent} is correct, then there are infinitely many bifurcations as $a\to0$ and there three different periods go to infinity. In contrast, the limiting geometric object is a square for which the symplectic billiard map is periodic with one period being 4.
}	
\end{remark}


\begin{figure}[t]
  \newcommand{\figheight}{0.491\linewidth}%
  \centering%
  \subfloat[\label{fig:house_space-tall}]{%
    \includegraphics[height=\figheight]{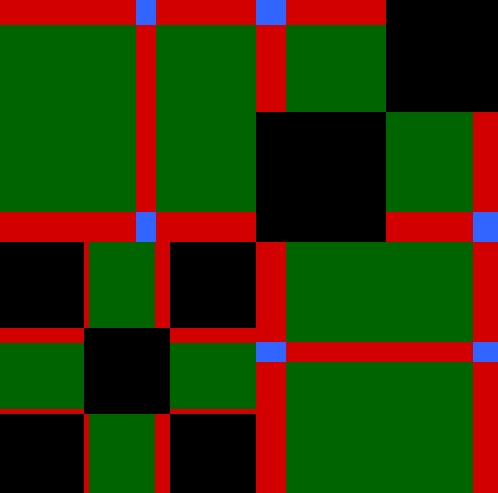}%
  }%
  \hfill%
  \subfloat[\label{fig:house_space-bifurcation}]{%
    \includegraphics[height=\figheight,trim={0 0 5pt 5pt},clip]{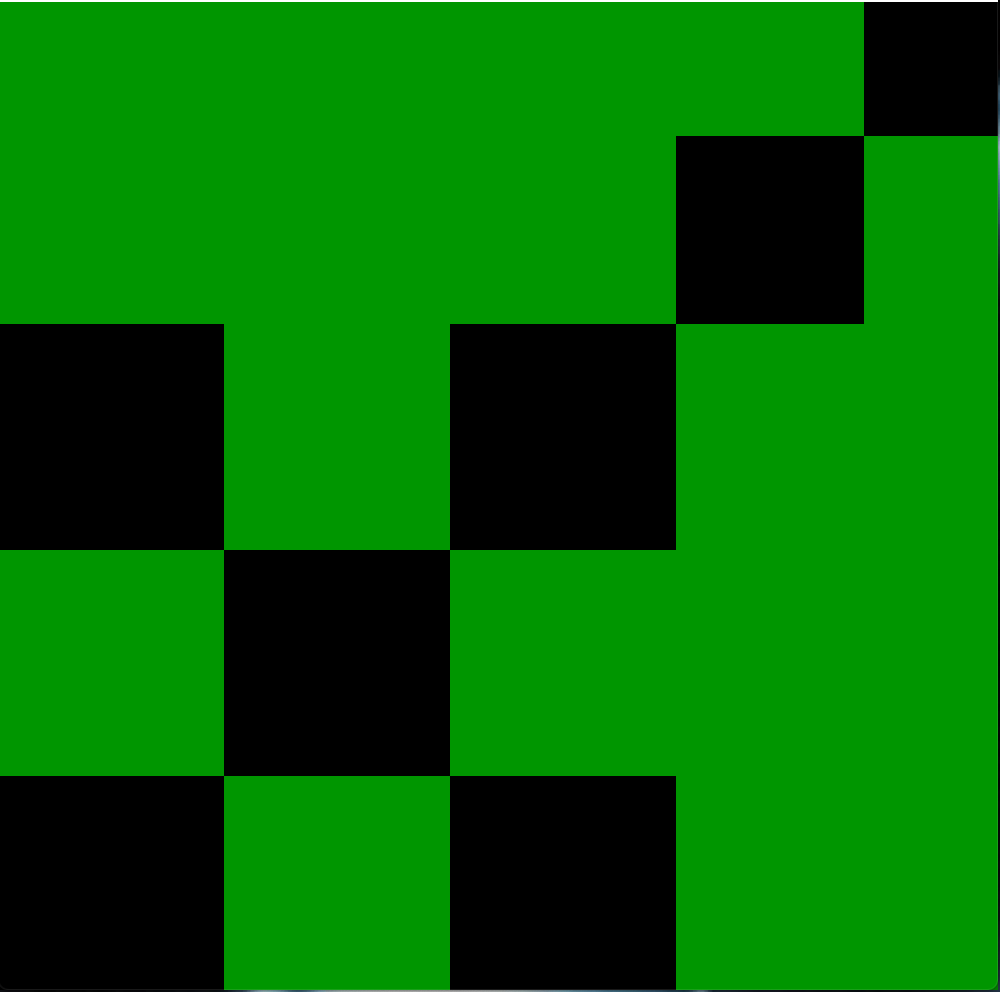}%
  }%
  \caption{\label{house_space}
    The phase space of a tall \penthouseex~\protect\subref{fig:house_space-tall} and at the first bifurcation case $a=1$~\protect\subref{fig:house_space-bifurcation}. The blue points have period 12, the red ones period 20, and the green ones period 28.
  }
\end{figure}

For a tall \penthouseex, we provide a complete analysis of the dynamics similar to the one given for the \quadex. We also examine the first bifurcation case, i.e., $a=1$.

\begin{theorem} \label{thm:pent}
All orbits in a tall \penthouseex are periodic with periods 12, 20, and 28. The structure of the orbits of periodic tiles is as follows, see \cref{house_space}.
\begin{itemize}
\item One orbit consisting of the tiles that return to themselves after 3 iterations, with the return map having 
order 4.
\item One orbit consisting of the tiles that return to themselves after 10 iterations, with the return map having 
order 2.
\item One orbit consisting of the tiles that return to themselves after 20 iterations, with the return map being the identity.
\item One orbit consisting of the tiles that return to themselves after 7 iterations, with the return map having 
order 4.
\item Two orbits consisting of the tiles that return to themselves after 28 iterations, with the return map being the identity.
\end{itemize}
In the first bifurcation case when $m$ changes values from $1$ to $2$, i.e., for $a=1$, the billiard map is fully periodic with one period being $28$, see \cref{house_space}. This is consistent with \cref{conj:pent}. We point out that there are three different types of symbolic orbits, though.
\end{theorem}

\begin{proof}
The following proof relies on \cref{house_1}. Without loss of generality, we assume that $b > 1/2$, that is, vertex $E$ is right of the perpendicular bisector of side $BC$. In this figure, the lines are parallel to the sides of the pentagon. If vertex $E$ moves horizontally, the order of some points on side $BC$ may change.

\begin{figure}[t]
\centering%
\includegraphics[height=3.3in]{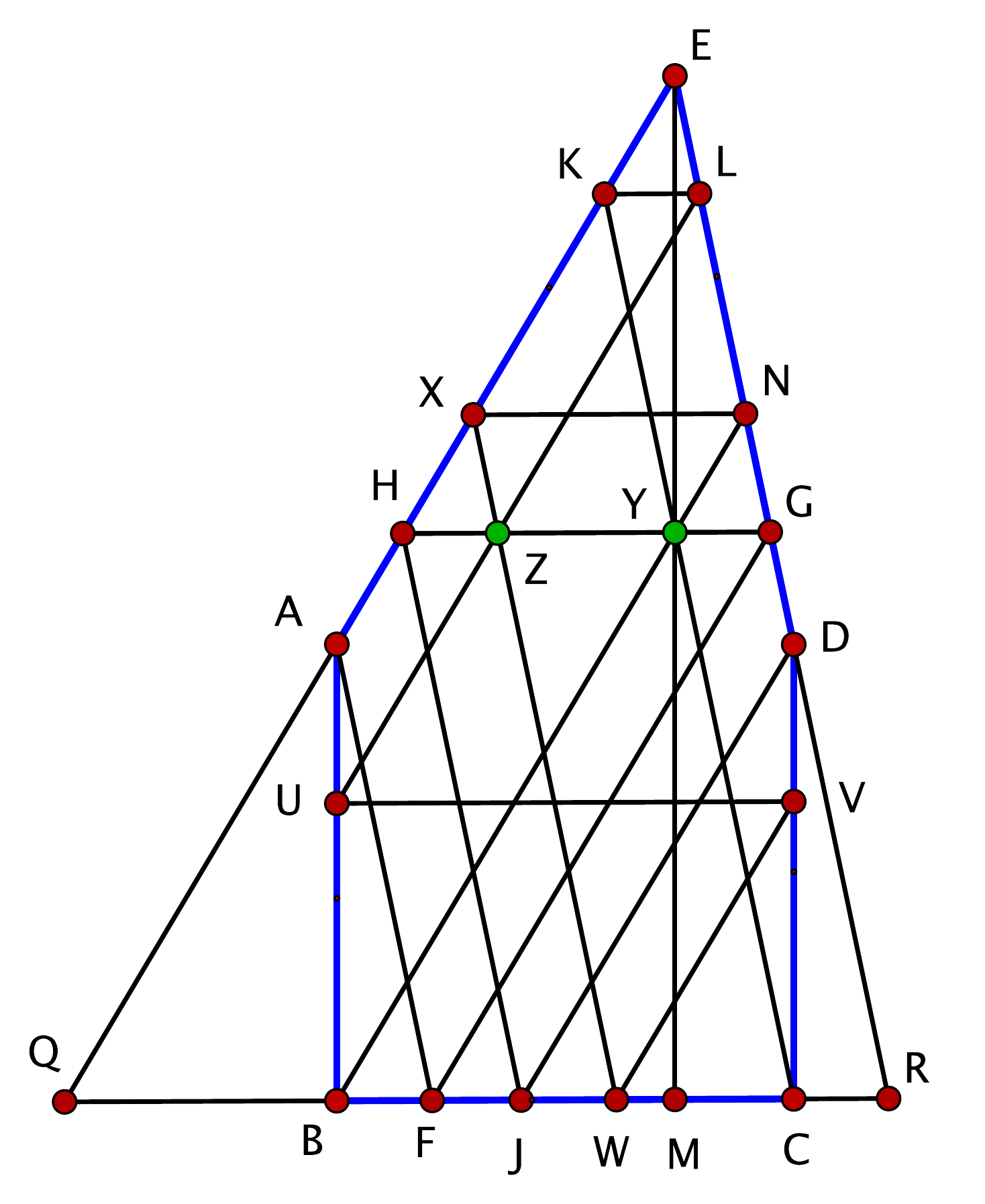}%
\caption{\label{house_1}
Marked tall \penthouseex.
}
\end{figure} 

For our analysis, it is important to notice two concurrences of lines, at points $Y$ and $Z$. Let us prove this elementary geometry fact for point $Y$; point $Z$ is treated similarly. We also note that 
$EK=XH,\ EL=NG$, and hence $HZ=YG$; this will follow from the analysis of the dynamics below.

Let $Y$ be the intersection point of the lines $BN$ and $CK$. First we show that the perpendicular, dropped from $E$ to the base $BC$, passes through $Y$. Indeed, since $ABCD$ is a square, the triangle $AED$ is obtained from the triangle $BYC$ by the vertical parallel translation. Therefore the altitude from vertex $Y$ is translated to the altitude from vertex $E$, hence they lie on the same vertical line.

Next, let $HG$ be the horizontal segment though point $Y$. 
Draw the line parallel to $EA$ through point $G$ to construct point $F$, and then the line parallel to $ED$ through point $F$.  We need to show that this line passes though the vertex of the square, point $A$. 

Indeed, let $A'$ be the intersection of this line with the vertical line though point $B$. We want to show that $A'=A$. We have $BF=YG=CR$, hence the triangles $CDR$ and $BA'F$ are congruent, and therefore $A'=A$. 

A similar argument shows that if one draws the line parallel to $ED$ through point $H$ to construct point $J$, and then the line parallel to $EA$ through point $J$, then this line passes through vertex $D$. 

Now we can describe the evolution of phase rectangles. We  refer to \cref{house_1}. For a point on the ``roof", such as point $X$, we use the notation $X^\perp$ for its orthogonal projection on the base $BC$ (these projections are not marked not to clutter the figure).

The 3-periodic orbit is easy to describe, it consists of the tiles surrounding the 3-periodic orbit in triangle $QER$ that connects the midpoints of its sides:
\begin{equation}\label{eqn:3periodic}
HA\times FJ \to FJ \times DG \to DG\times HA \to HA\times FJ.
\end{equation}

The 10-periodic orbit is as follows:
\begin{equation}
\begin{aligned} \label{eqn:10periodic}
NL\times HA \to HA\times UB \to UB\times BH^\perp \to BH^\perp \times CV \to CV \times HA \to\\ HA \times WC \to WC\times DG \to DG \times KX \to KX\times FJ \to FJ\times NL \to NL\times HA.
\end{aligned}
\end{equation}
The 7-periodic orbit is as follows:
\begin{equation} \label{eqn:7periodic}
\begin{aligned}
NL\times KX \to KX\times UB \to UB\times X^\perp K^\perp \to X^\perp K^\perp \times CV \to\\
 CV \times KX \to KX \times WC \to WC\times NL \to NL\times KX.
\end{aligned}
\end{equation}
Here is a 20-periodic orbit:
\begin{equation} \label{eqn:20periodic}
\begin{aligned}
&LE\times HA \to HA\times AU \to AU\times BH^\perp \to BH^\perp \times VD \to VD\times HA \to 
\\
 &HA\times JW \to JW\times DG \to DG\times XH \to XH\times FJ \to FJ\times GN \to 
 \\
 &GN\times HA \to HA\times BF \to BF\times DG \to DG\times AB \to AB\times G^\perp C \to 
 \\
 &G^\perp C\times CD \to CD\times DG \to DG\times EK \to EK\times FJ \to FJ\times LE \to LE\times HA.
\end{aligned}
\end{equation}
It remains to describe the two 28-periodic orbits. Here they are:
\begin{equation} \label{eqn:28periodic}
\begin{aligned}
LE\times EK \to EK\times AU \to AU\times K^\perp E^\perp \to K^\perp E^\perp \times VD \to VD\times EK \to \\
EK\times JW \to JW\times LE \to LE\times XH \to XH\times AU \to AU\times H^\perp X^\perp \to \\
H^\perp X^\perp \times VD \to VD \times XH \to XH\times JW \to JW\times GN \to GN \times XH \to \\
XH\times BF \to BF\times GN \to GN\times AB \to AB \times N^\perp G^\perp  \to N^\perp G^\perp \times CD \to \\
CD \times GN \to GN \times EK \to EK\times BF \to BF\times LE \to LE\times AB \to \\
AB\times E^\perp L^\perp \to E^\perp L^\perp \times CD \to CD\times LE \to LE\times EK,
\end{aligned}
\end{equation}
and
\begin{equation} \label{eqn:28periodic_2}
\begin{aligned}
NL\times EK \to EK\times UB \to UB\times K^\perp M \to K^\perp M\times CV \to CV\times EK \to
\\
EK\times WC \to WC \times LE \to LE\times KX \to KX\times AU \to AU\times X^\perp K^\perp \to 
\\
X^\perp K^\perp \times VD \to VD\times KX \to KX\times JW \to JW\times NL \to NL \times XH \to
\\
XH\times UB \to UB\times H^\perp X^\perp \to H^\perp X^\perp \times CV \to CV\times XH \to XH\times WC \to
\\
WC\times GN \to GN\times KX \to KX\times BF \to BF\times NL \to NL\times AB \to
\\
AB\times L^\perp N^\perp \to L^\perp N^\perp \times CD \to CD\times NL \to NL\times EK.
\end{aligned}
\end{equation}

By inspection, the above described orbits cover the whole phase space. Therefore, it remains to consider the case $a=1$.

In the limit $a\searrow1$ the tiles with sides $HA$, $FJ$, and $DG$ disappear simultaneously. This kills the periodic orbits \eqref{eqn:3periodic}, \eqref{eqn:10periodic}, and \eqref{eqn:20periodic}. At the same time the periodic orbits \eqref{eqn:7periodic}, \eqref{eqn:28periodic}, and \eqref{eqn:28periodic_2} survive. Their tiles still cover the phase space and they give rise to different symbolic orbits.
\end{proof}

The next \cref{pent_perturb} shows the phase space of a small perturbation of a tall \penthouseex.

\begin{figure}[t]
\centering%
\includegraphics[height=4in,trim={1pt 0 2pt 0},clip]{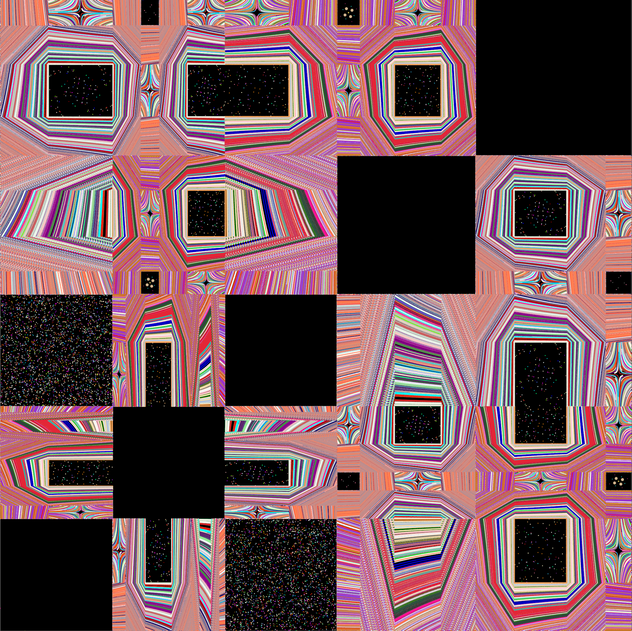}%
\caption{\label{pent_perturb}
The phase portrait of a small perturbation of a tall \penthouseex: in accordance with \cref{prop:periodic}, the 3- and 7-periodic orbits survive the perturbation.
}
\end{figure} 

\section{(Lattice) hexagons with parallel opposite sides} \label{sect:hexpar}

We first consider hexagons with parallel opposite sides. After an affine transformation we assume that the directions of the sides are those of an equilateral triangle. In this situation we have the following immediate corollary of \cref{lm:area}.

\begin{figure}[t]
  \newcommand{\figheight}{0.45\linewidth}%
  \centering%
  \subfloat[\label{fig:hexpar}]{%
    \includegraphics[height=\figheight,trim={3pt 7pt 33pt 2pt},clip]{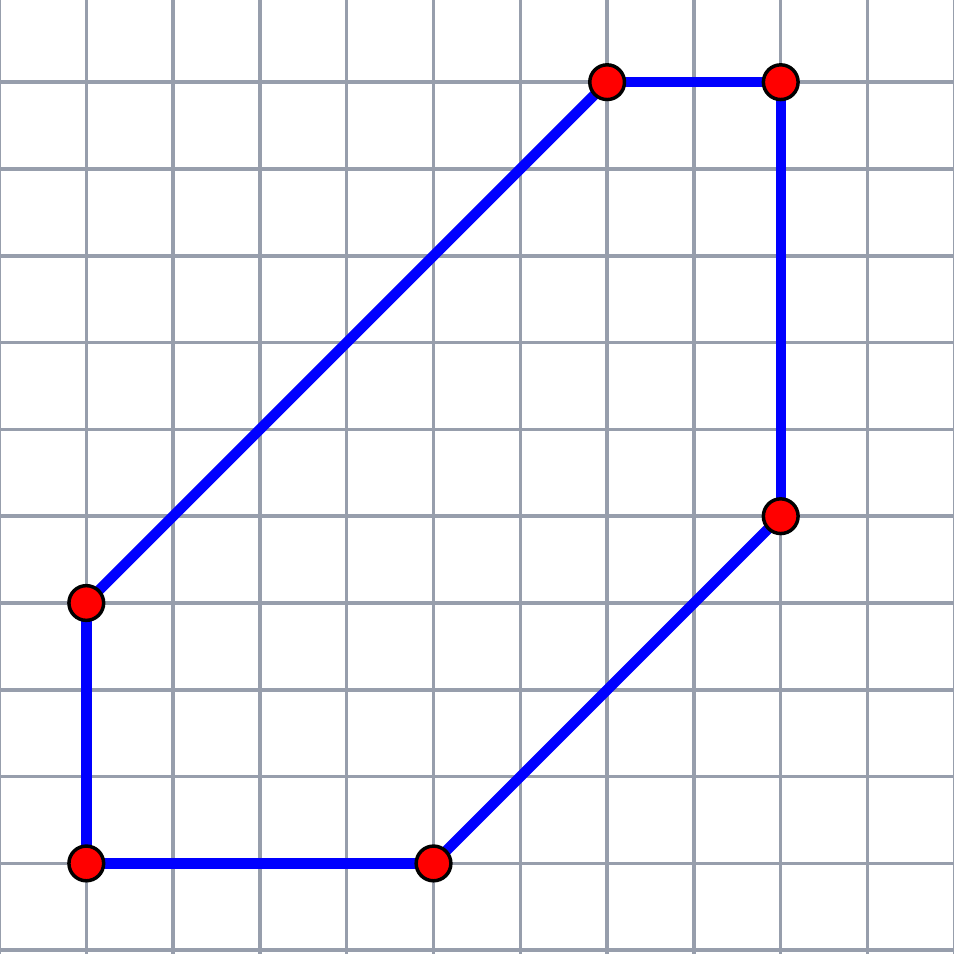}%
  }%
  \hfil%
  \subfloat[\label{fig:hexpar-orbit}]{%
    \includegraphics[height=\figheight]{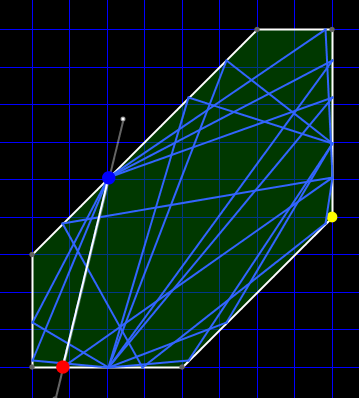}%
  }%
  \caption{\label{hexpar_1}
  A lattice hexagon with parallel opposite sides~\protect\subref{fig:hexpar} and a corresponding periodic orbit ~\protect\subref{fig:hexpar-orbit}.
  }
\end{figure}

\begin{corollary}\label{cor:hexagon_from_equilateral_triangle}
Let $P$ be a hexagon as above. The symplectic billiard map $T$ is a local Euclidean isometry and has the form
$$
T: (x,y) \mapsto (y, -x+b),\ \ x\in P_i P_{i+1}, y\in P_j P_{j+1}, z\in P_k P_{k+1},
$$
where $b$ depends on $i,j,k$. In particular, the symplectic billiard map $T$ has no hyperbolic periodic points.
\end{corollary}

\begin{remark}
{\rm
The affine invariant version of the previous corollary is that for any hexagon with parallel opposite sides the symplectic billiard map is a local  isometry with respect to inner product $AA^t$, where $A$ relates the given hexagon with a hexagon from \cref{cor:hexagon_from_equilateral_triangle}.

We also point out that, as opposed to \cref{rmk:symplectic_billiard_is_an_indefinite_isometry}, in this special situation $T$ is a isometry with respect to a sign-definite inner product, e.g., leading to the strong conclusion that it doesn't admit hyperbolic periodic points.
}
\end{remark}

A hexagon with parallel opposite sides is obtained from a triangle by cutting off the corners by the lines parallel to the sides. Applying an affine transformation, we may assume that the slopes of the sides are equal to $0, 1$ and $\infty$. In addition, we now consider lattice polygons, i.e, all vertices are lattice points, see \cref{hexpar_1}. The affine moduli space of these hexagons is described by three integral parameters (whereas the affine moduli space of hexagons is 6-dimensional).


The lattice points partition the sides into segments; let $p_1,q_1,r_1,p_2,q_2,r_2$ be the number of these elementary segments on the six sides in the cyclic order. Set
$$
N = p_1q_1+q_1r_1+r_1p_2+p_2q_2+q_2r_2+r_2p_1+
p_1r_1+q_1p_2+r_1q_2+p_2r_2+q_2p_1+r_2q_1.
$$

\begin{theorem} \label{thm:parhex}
All orbits in a lattice hexagon with parallel opposite sides are periodic, and the periods do not exceed $4N$.
\end{theorem}

\begin{proof}
The phase space is subdivided into the tiles formed by the products of the elementary segments on the sides, that is, segments between lattice points, see \cref{hexpar_1}. These tiles evolve as single pieces under the symplectic billiard map. The phase area of each tile is one, and the whole phase area equals $N$. By the area preserving property, the period  of the orbit of each tile does not exceed $N$, and the return map to a tile is at most 4-periodic. 
\end{proof}

Of course, the upper bound of this theorem is unrealistically high. See \cref{hexpar_1} for an example of an orbit  and \cref{hex_perturb} for the phase space, colored according to  periods.


The next \cref{hex_perturb} shows the phase portrait of a small perturbation of a hexagon with parallel opposite sides.


\begin{figure}[t]
  \newcommand{\figheight}{0.493\linewidth}%
  \centering%
  \subfloat[\label{fig:hexpar-phase}]{%
    \includegraphics[height=\figheight,trim={28pt 0 0 0},clip]{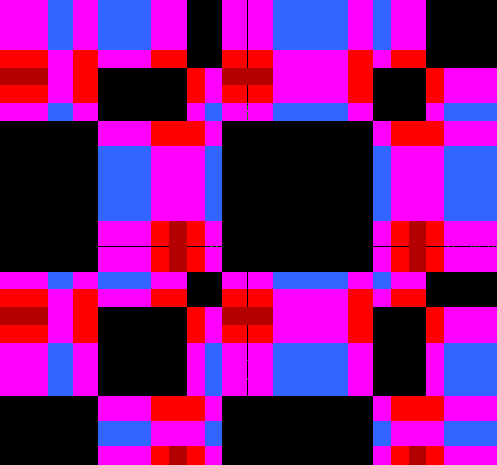}%
  }%
  \hfill%
  \subfloat[\label{fig:hexpar-perturbed}]{%
    \includegraphics[height=\figheight,trim={131pt 131pt 0 0},clip]{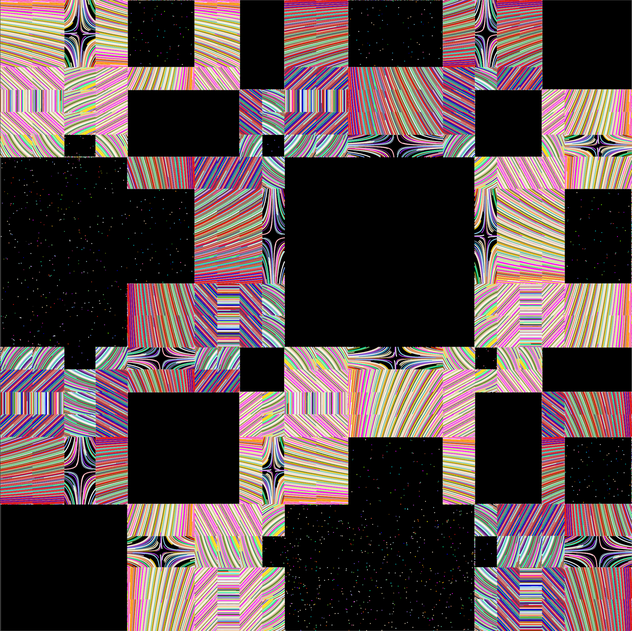}%
  }%
  \caption{\label{hex_perturb}
  The phase space~\protect\subref{fig:hexpar-phase} of the hexagon from \cref{hexpar_1}: the periods are $4,12,24, 36$. The phase space~\protect\subref{fig:hexpar-perturbed} of a small perturbation of the same hexagon.
  }
\end{figure}

\begin{remark}
{\rm
Consider a hexagon $P$ with parallel opposite sides having rational slopes. One can approximate $P$ by a rational hexagon whose vertices have rational coordinates. The symplectic billiard orbits in the approximating polygons are periodic, but their periods will grow with the least common denominators of the coordinates of the vertices of the approximating polygons. 
The same applies to the polygons described in the next section.
}	
\end{remark}

\section{The Hex(en)house and special octagons} \label{sect:hexhouse}

In this section we present two more families with fully periodic symplectic billiard map, the hex(en)house and special octagons. We can only give computer evidence and formulate two corresponding conjectures.

\begin{figure}[t]
\centering%
\includegraphics[height=1.5in]{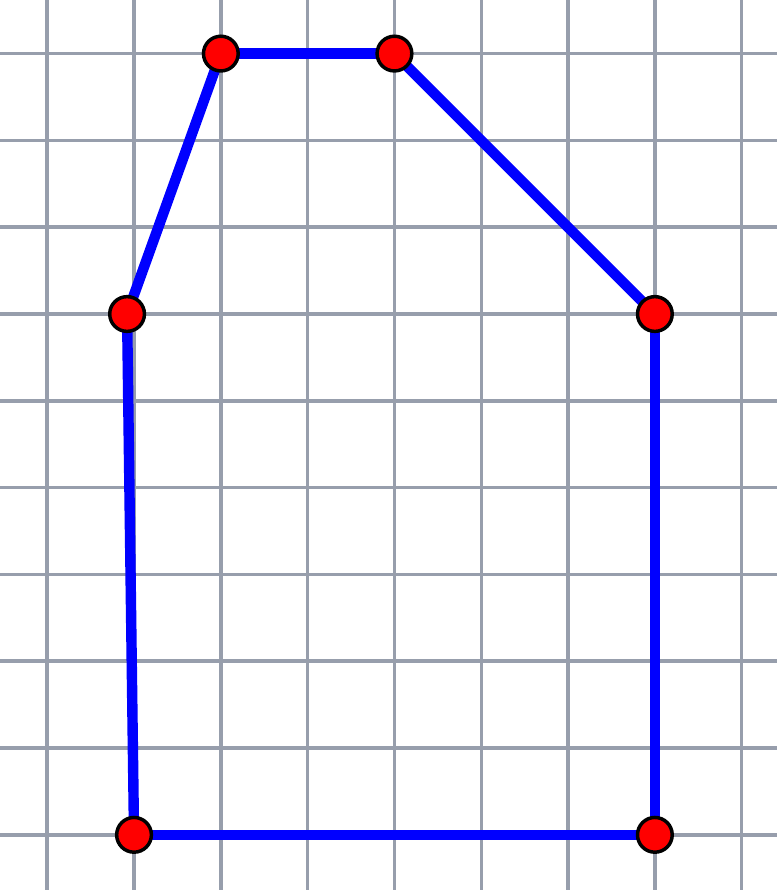}%
\caption{\label{hexh}
A hex(en)house.
}
\end{figure} 

The \hexhouseex is a lattice polygon obtained by placing a trapezoid on top of a square, see \cref{hexh}. The moduli space of such polygons is described by three integral parameters.
%
%
\cref{fig:hexhouse-list-1,fig:hexhouse-list-2,fig:hexhouse-list-3,fig:hexhouse-list-4} support the following conjecture. \cref{fig:hexhouse-perturbed} shows again the typical pattern of a perturbation.

\begin{conjecture} \label{conj:hexhouse}
All orbits in a \hexhouseex are periodic.
\end{conjecture}

\begin{figure}[t]
  \newcommand{\figheight}{0.49\linewidth}%
  \centering%
  \subfloat[]{%
    \includegraphics[height=\figheight,trim={0 0 35.7cm 0},clip]{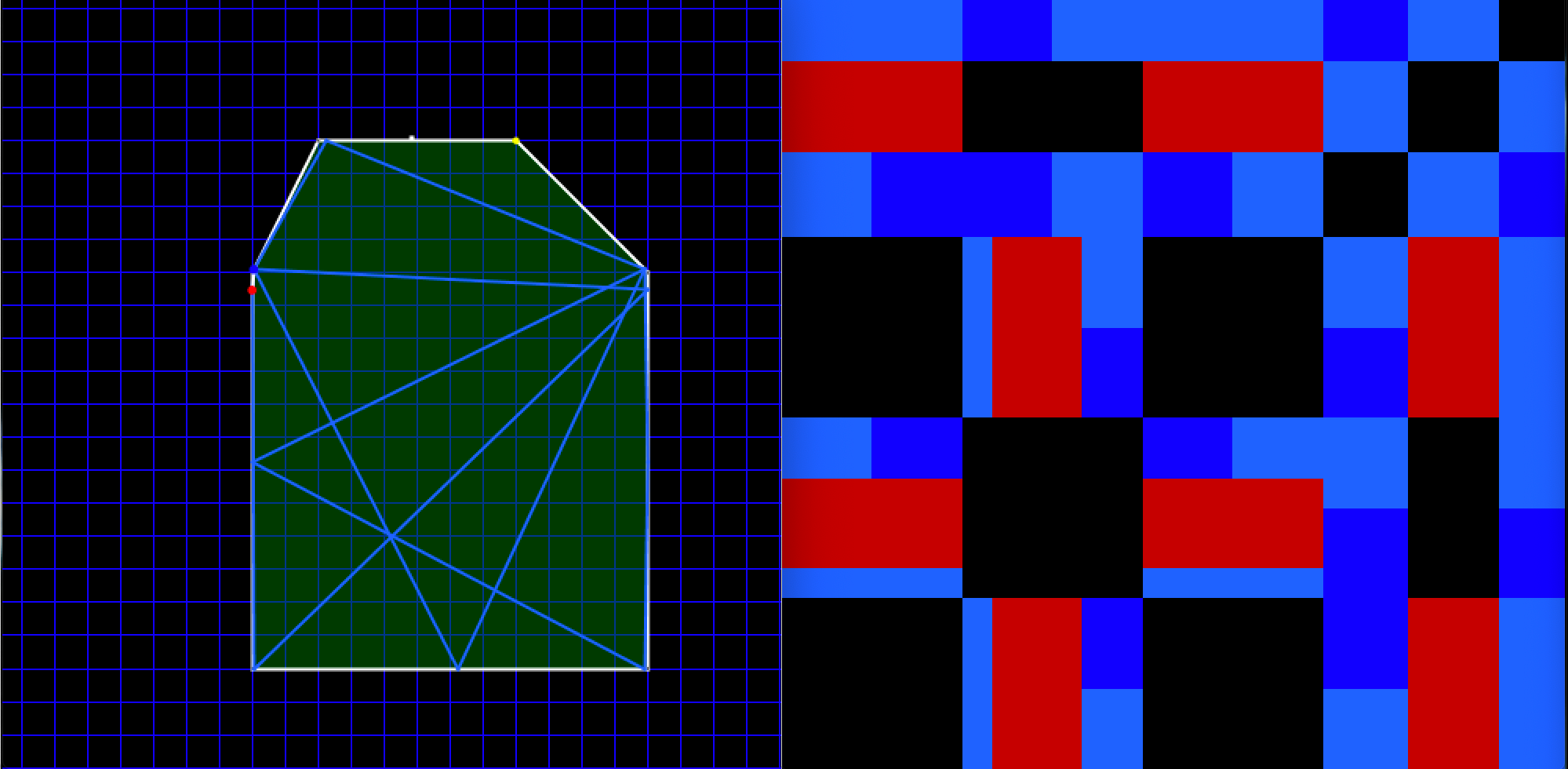}%
  }%
  \hfill%
  \subfloat[]{%
    \includegraphics[height=\figheight,trim={35.2cm 0 5pt 0},clip]{figures/hexhouse_periodic_1_flipped}%
  }%
  \caption{\label{fig:hexhouse-list-1}
    A periodic \hexhouseex with periods 4, 12, 28.
  }
\end{figure}

\begin{figure}[t]
  \newcommand{\figheight}{0.494\linewidth}%
  \centering%
  \subfloat[]{%
    \includegraphics[height=\figheight,trim={0 0 35.7cm 0},clip]{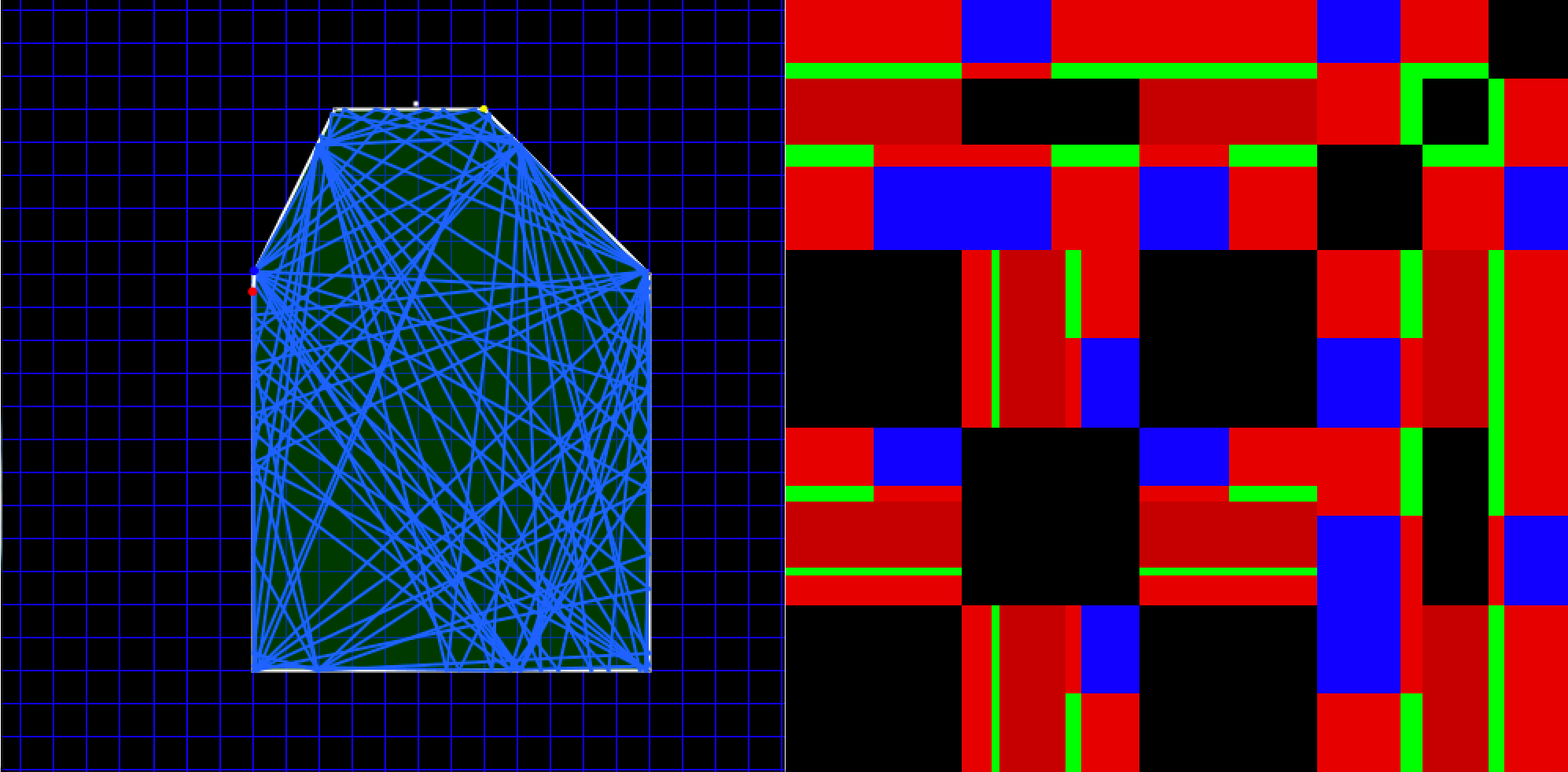}%
  }%
  \hfill%
  \subfloat[]{%
    \includegraphics[height=\figheight,trim={35.3cm 0 5pt 0},clip]{figures/hexhouse_periodic_2_flipped}%
  }%
  \caption{\label{fig:hexhouse-list-2}
    Another periodic \hexhouseex with periods 4, 28, 108, 188.
  }
\end{figure} 

\begin{figure}[t]
  \newcommand{\figheight}{0.494\linewidth}%
  \centering%
  \subfloat[]{%
    \includegraphics[height=\figheight,trim={0 0 37.7cm 2cm},clip]{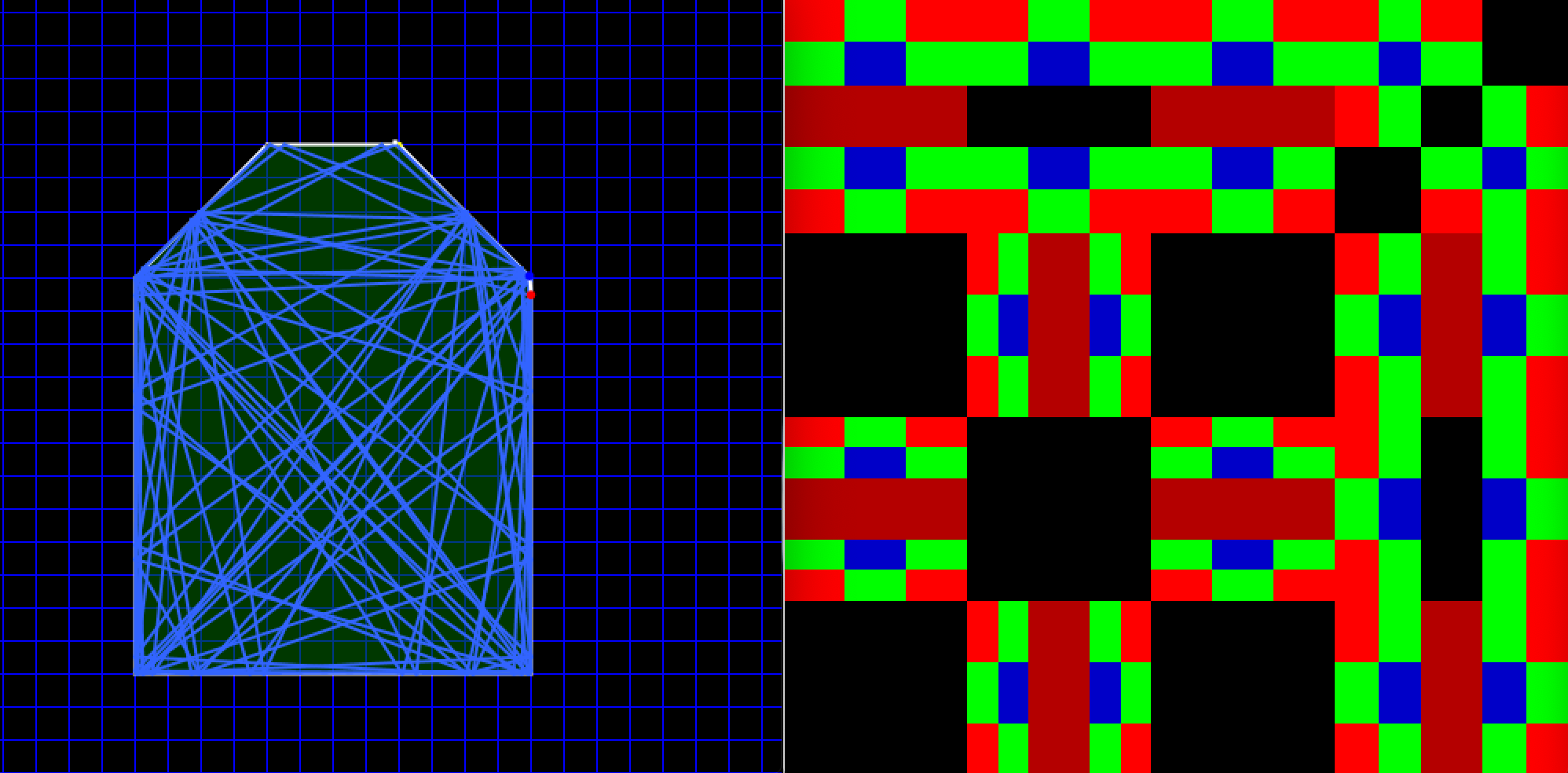}%
  }%
  \hfill%
  \subfloat[]{%
    \includegraphics[height=\figheight,trim={35.3cm 0 5pt 0},clip]{figures/hexhouse_periodic_3}%
  }%
  \caption{\label{fig:hexhouse-list-3}
    Yet another periodic \hexhouseex with periods 4, 44, 68, 92.
  }
\end{figure} 

\begin{figure}[t]
  \newcommand{\figheight}{0.4992\linewidth}%
  \centering%
  \subfloat[]{%
    \includegraphics[height=\figheight,trim={3cm 0 35.7cm 2cm},clip]{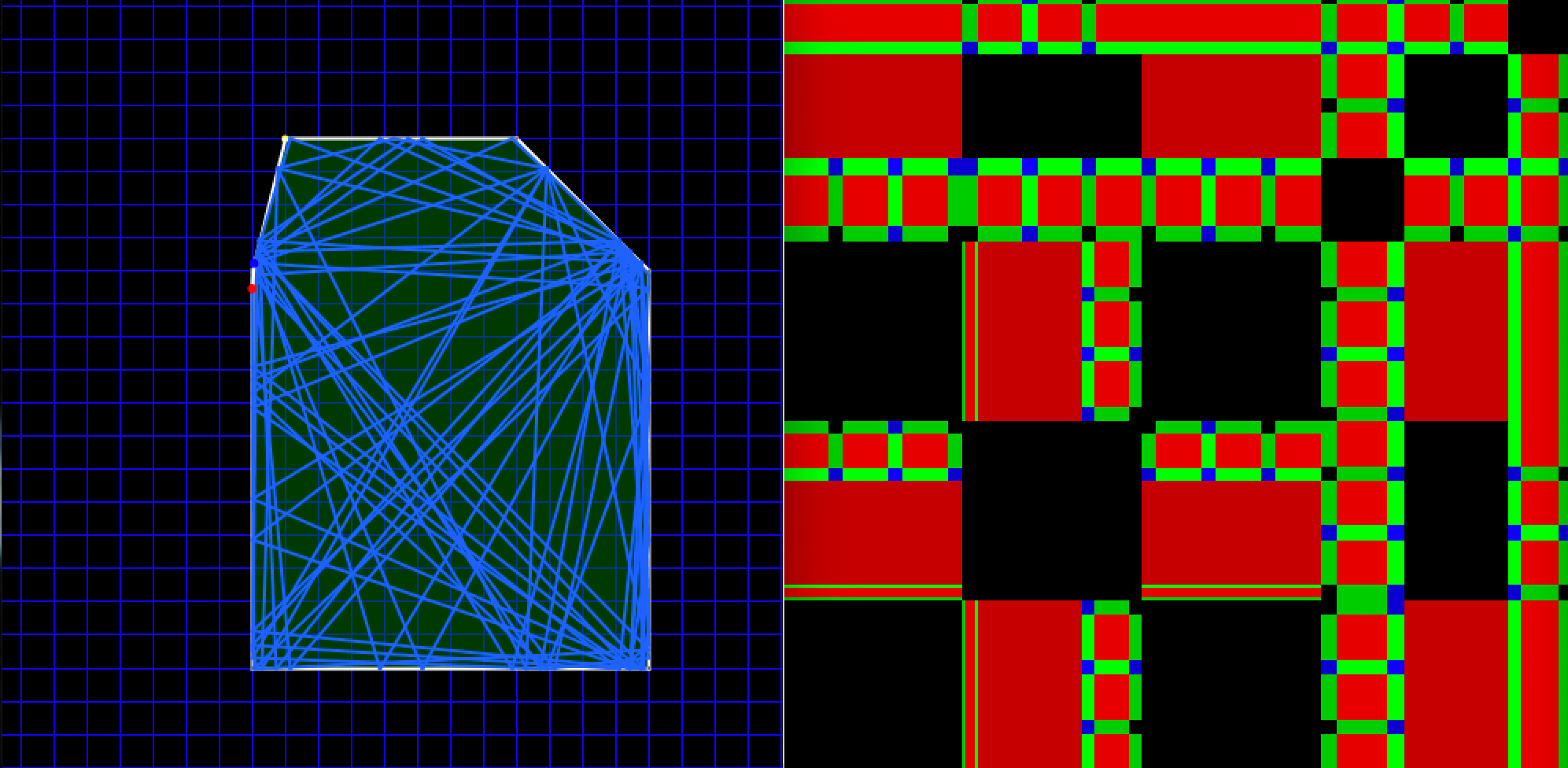}%
  }%
  \hfill%
  \subfloat[]{%
    \includegraphics[height=\figheight,trim={35.3cm 0 6pt 0},clip]{figures/hexhouse_periodic_4_flipped}%
  }%
  \caption{\label{fig:hexhouse-list-4}
    The last periodic \hexhouseex, periods are 4, 28, 44, 60, 68, 84, 108.
  }
\end{figure} 

\begin{figure}[t]
\centering%
\includegraphics[height=4in,trim={0 0 5pt 7pt},clip]{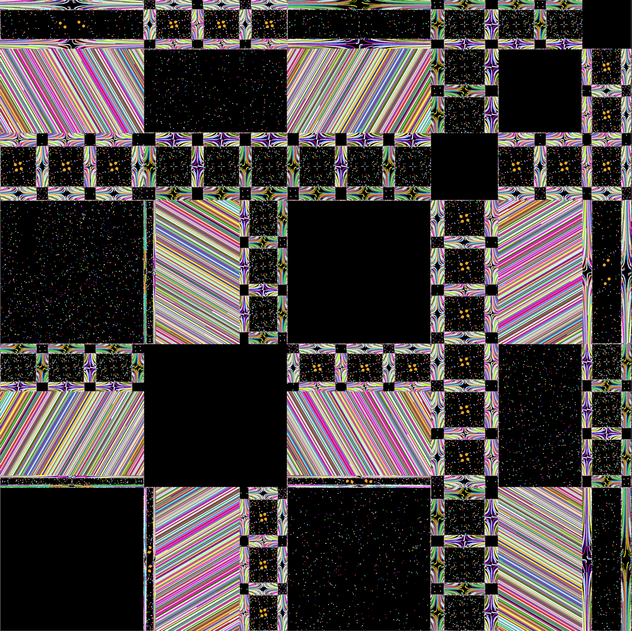}%
\caption{\label{fig:hexhouse-perturbed}
A small perturbation of the previous \hexhouseex.
}
\end{figure} 

\clearpage

Special octagons are lattice octagons whose opposite sides are parallel and have slopes $0,\pm1, \infty$, and that 
 have an axis of symmetry parallel to a pair of sides, see \cref{octapar}. Their symmetries are dilations and translations in $\mathbf{R}^2$. The corresponding moduli space is 3-dimensional over the integers since we only consider lattice octagons. This can be seen as follows. Fix a rectangle with vertical and horizontal sides centered at the origin. Then cut of two corners with diagonal lines and cut the other two corners according to symmetry. These are four-dimensional choices which after dividing out dilations give rise to a three dimensional moduli space.

\begin{figure}[t]
  \centering%
\includegraphics[height=2in]{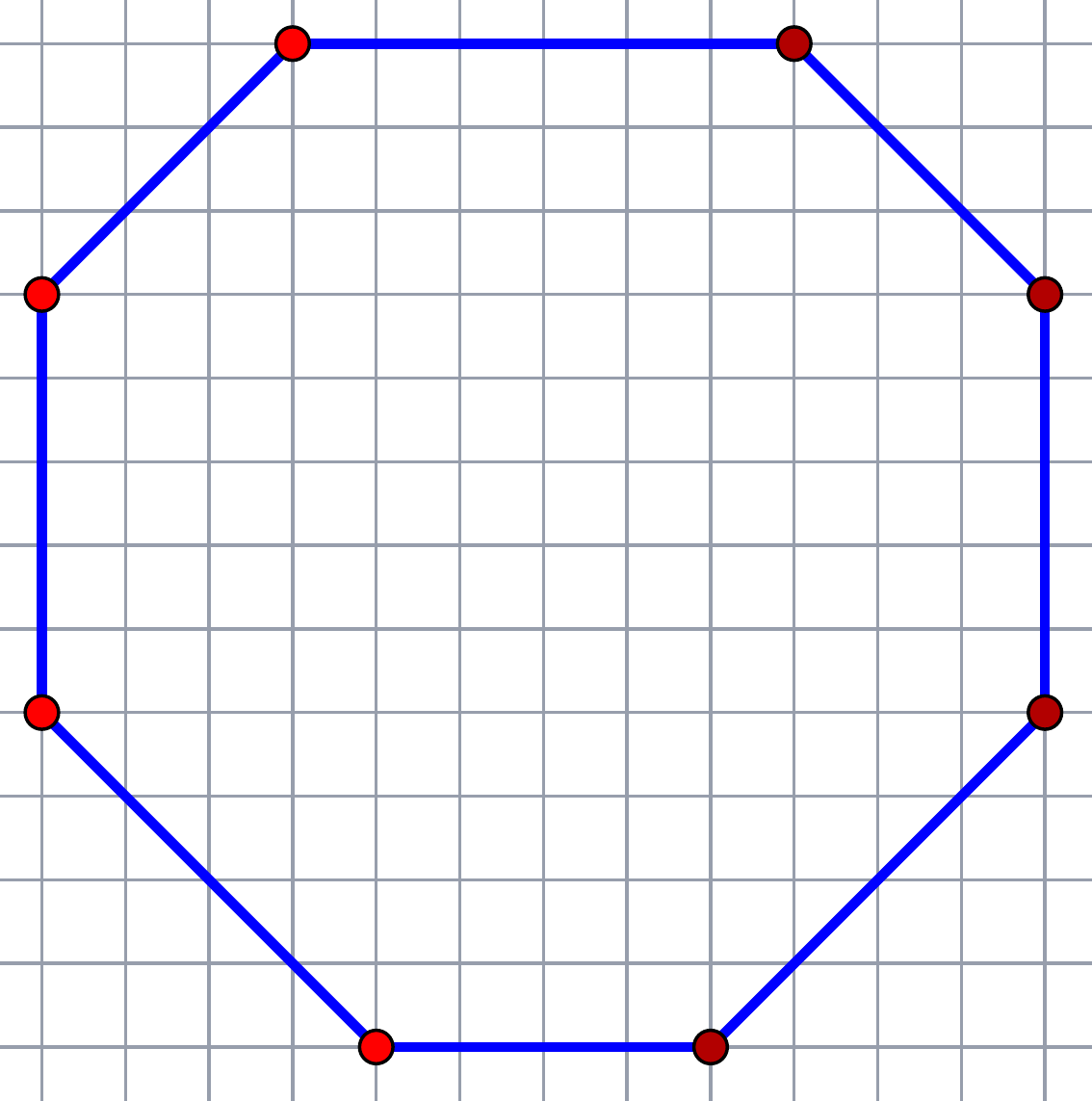}%
\caption{\label{octapar}
A special octagon. 
}
\end{figure} 

\cref{fig:special-octagon-1,fig:special-octagon-2} show two pictures as evidence for the following conjecture, and \cref{fig:special-octagon-perturbed} again the typical pattern caused by a perturbation.

\begin{conjecture} \label{conj:special_octagons}
All orbits in a special octagon are periodic.
\end{conjecture}

\begin{figure}[t]
  \newcommand{\figheight}{0.494\linewidth}%
  \centering%
  \subfloat[]{%
    \includegraphics[height=\figheight,trim={0 0 17.6cm 0},clip]{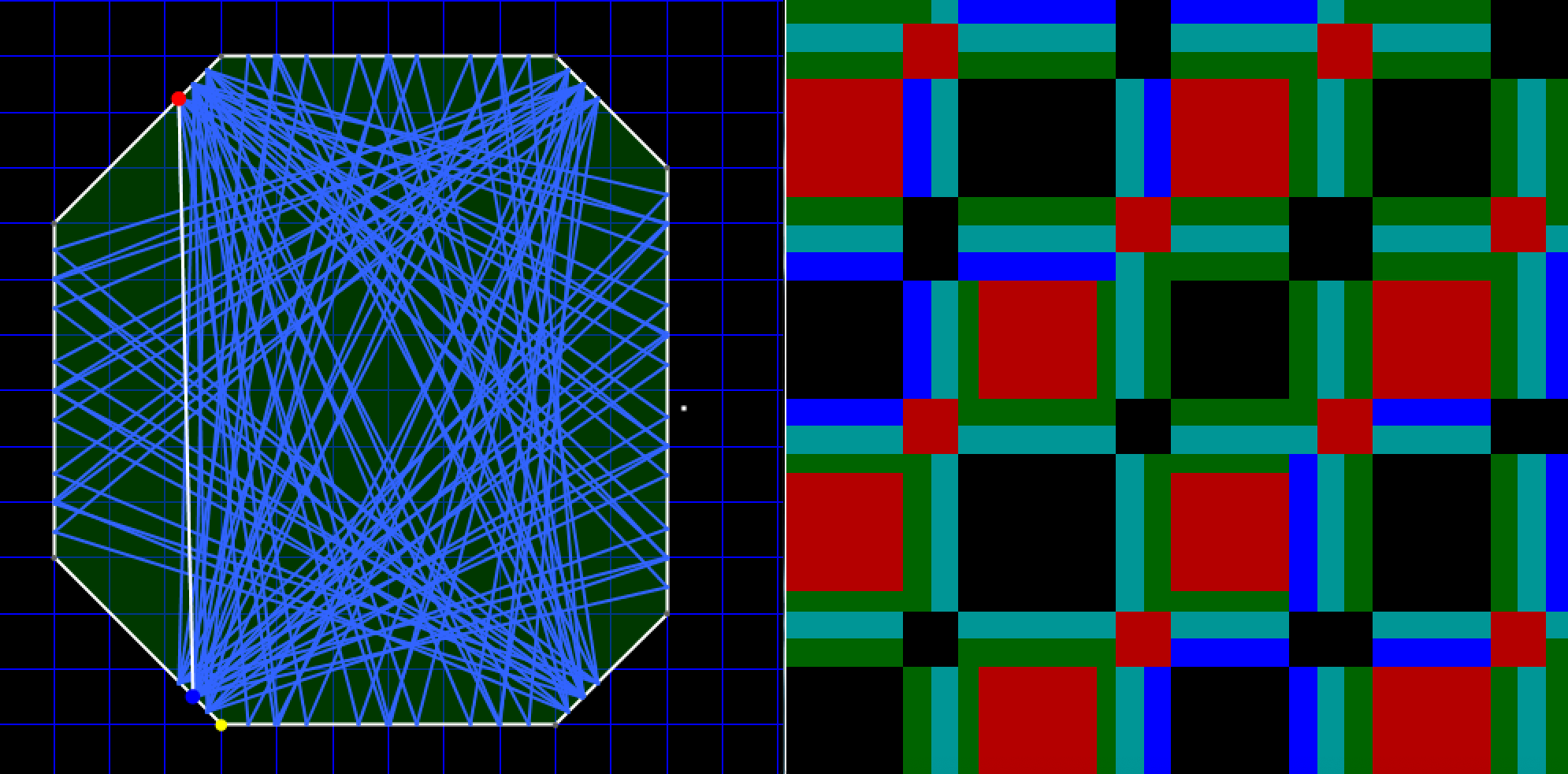}%
  }%
  \hfill%
  \subfloat[]{%
    \includegraphics[height=\figheight,trim={17.85cm 0 0 0},clip]{figures/octa_periodic_1}%
  }%
  \caption{\label{fig:special-octagon-1}
    A periodic special octagon with periods: 4, 56, 68, 108. 
}
\end{figure} 

\begin{figure}[t]
  \newcommand{\figheight}{0.494\linewidth}%
  \centering%
  \subfloat[]{%
    \includegraphics[height=\figheight,trim={0 0 18cm 0},clip]{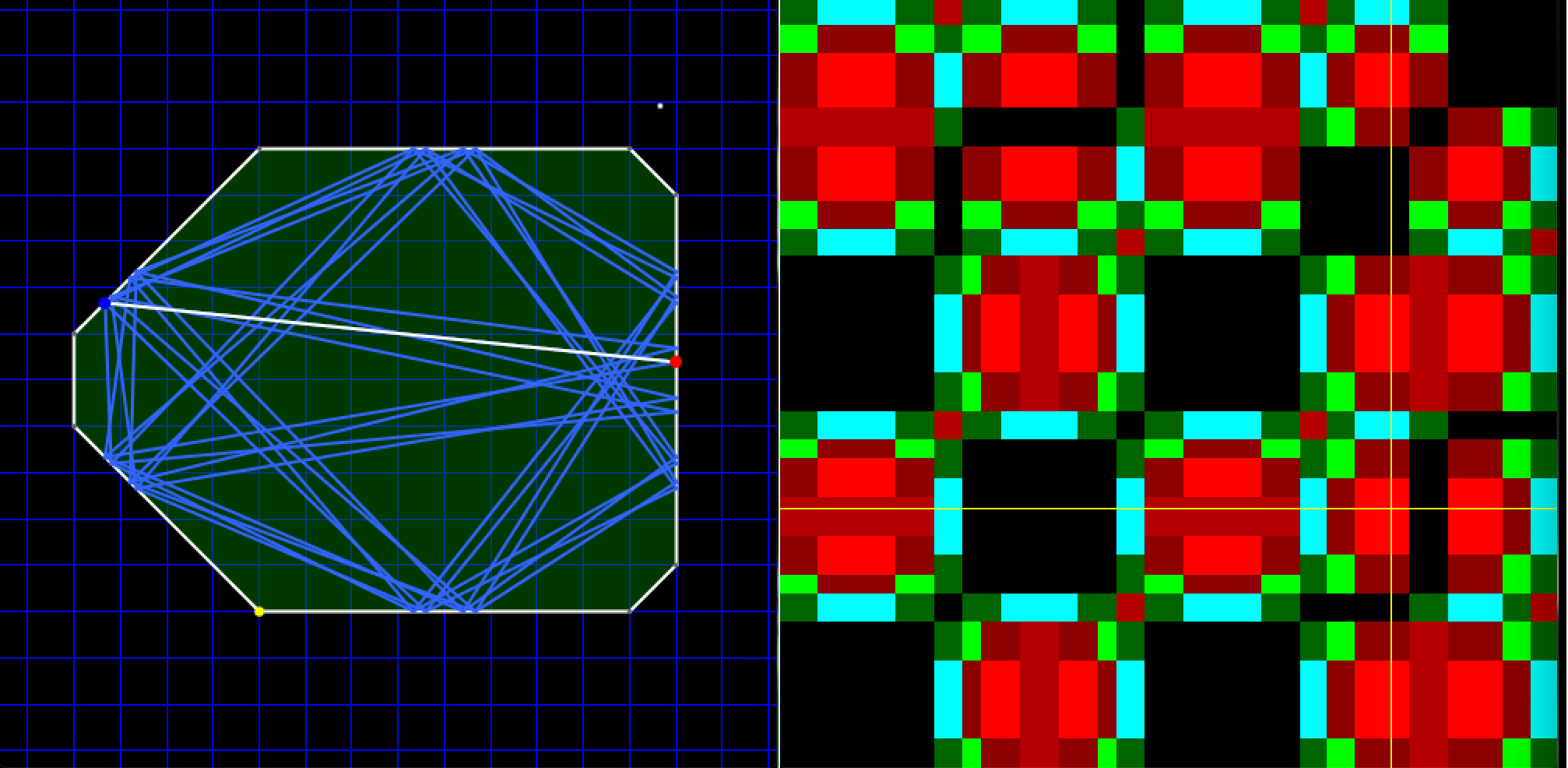}%
  }%
  \hfill%
  \subfloat[]{%
    \includegraphics[height=\figheight,trim={17.9cm 0 8pt 0},clip]{figures/octa_periodic_2}%
  }%
\caption{\label{fig:special-octagon-2}
A second periodic special octagon with periods: 4, 16, 32, 44, 68, 92.
}
\end{figure} 

\begin{figure}[t]
\centering%
\includegraphics[height=4in,trim={3pt 0 5pt 0},clip]{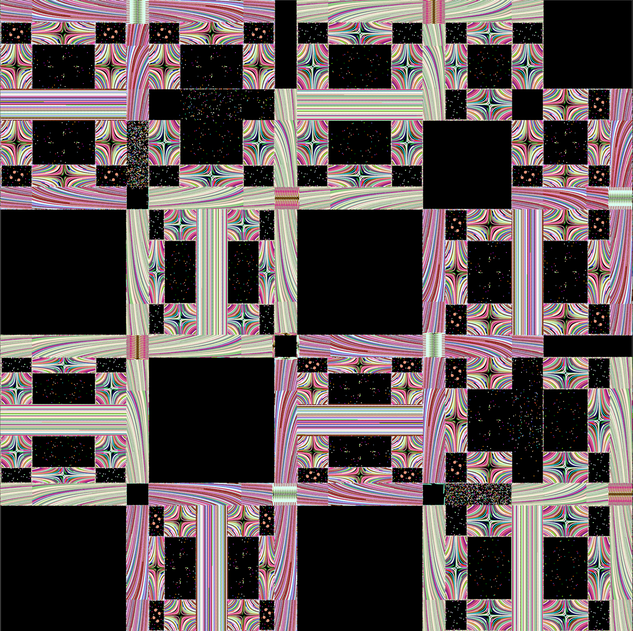}%
\caption{\label{fig:special-octagon-perturbed}
A small perturbation of the previous special octagon.
}
\end{figure} 

\section{Open problems and conjectures} \label{sect:problems}

An outstanding open problem concerning (Euclidean) polygonal billiards is whether they always have a periodic orbit. This is not known even for obtuse triangles (the acute and right triangles possess periodic orbits). The current state of the art is that all obtuse triangles with the angles not exceeding $112.3$ degrees have periodic billiard trajectories, \cite{Sch1,To}. In contrast, all polygonal outer billiards possess periodic orbits, \cite{Ta1}.

\begin{question}
Do all polygonal symplectic billiards have periodic trajectories?
\end{question}

This is particularly intriguing since a computer search on the kites with corners given by $(-1,1),(-1,-1),(1,-1),(3,3)$  did not find a periodic orbit of period less than 2000. This is the ``smallest'' lattice kite not being a square. 

\bigskip
\subsection*{Acknowledgements} This work is supported by Deutsche Forschungsgemeinschaft (DFG) under Germany's Excellence Strategy EXC-2181/1 - 390900948 (the Heidelberg STRUCTURES Excellence Cluster) and by the Transregional Colloborative Research Center SFB / TRR 191, NSF grant DMS-1510055, the Interdisciplinary Center for Scientific Computing (IWR), and HGS MathComp.

We also would like to thank Lutz Hofmann and Peter H\"ugel for their technical support and contributions, and respective funding within the subproject A7 of the Transregional Colloborative Research Center SFB / TRR 165.

\end{document}